\documentclass[11pt]{amsart}
\usepackage{hyperref}
\usepackage{mathtools}
\usepackage{amssymb}
\usepackage{color}
\usepackage{physics}
\usepackage[left=3cm,right=3cm,top=3cm,bottom=3cm]{geometry}
\usepackage{enumitem}

\newtheorem{theorem}{Theorem}

\newtheorem{lemma}{Lemma}
\newtheorem{proposition}{Proposition}
\theoremstyle{remark}
\newtheorem{remark}{Remark}

\newcommand\eps{\varepsilon}
\newcommand\N{\mathbb{N}}
\newcommand\R{\mathbb{R}}
\newcommand\Z{\mathbb{Z}}

\newcommand\ud{\, \textnormal{d}}
\newcommand\ii{\textnormal{i}}
\newcommand\UU{U}
\newcommand\Uk{U_{\alpha_k}}
\newcommand\pp{p}
\newcommand\PP{P}
\newcommand\qqq{q}

\DeclareMathOperator{\DIST}{dist}
\DeclareMathOperator{\SIGN}{sign}

\title[Small energy breathers of the nonlinear Klein-Gordon equation]
{Asymptotic analysis of small energy breathers \\ for the nonlinear Klein-Gordon equation}

\author{Micha{\l} Kowalczyk}
\address{Departamento de Ingeniería Matemática and Centro
de Modelamiento Matemático (UMI 2807 CNRS), Universidad de Chile, Casilla
170 Correo 3, Santiago, Chile.}
\email{kowalczy@dim.uchile.cl}

\author{Yvan Martel}
\address{Laboratoire de mathématiques de Versailles, UVSQ, Université Paris-Saclay, CNRS,
45 avenue des Etats-Unis, 78035 Versailles}
\email{yvan.martel@uvsq.fr}

\thanks{M.K. was partially funded by Chilean research grants FONDECYT 1250156, 1210405 
and ANID FB 210005. Both authors were supported by the ECOS-ANID project
ECOS 240024. Part of this work was done when Y.M. was visiting CMM, Universidad de Chile.
His visit was supported by the Chilean research grant FONDECYT 1210405.}

\begin{document}

\begin{abstract}
For a class of nonlinear Klein-Gordon equations,
we prove that in the small energy limit, any sequence of breathers
decomposes into a finite sum of decoupled, periodically modulated canonical solitons.
Each of these solitons is asymptotically equal to an explicit sine-Gordon breather and the distance between them grows to infinity 
as the energy decreases to 0. Finally we prove that none of these breathers is centered in a bounded set provided that 
a certain non resonance condition holds.
\end{abstract}

\maketitle

\section{Introduction}

\subsection{Main result}
We consider a class of nonlinear Klein-Gordon equations
\begin{equation}
\label{eq:nlkg}
\partial_t^2\phi = \partial_x^2 \phi -\phi + \UU(x)\phi^2 +\frac16 \phi^3 +p(\phi),
\quad (t,x)\in \R\times\R, \quad \phi(t,x)\in \R,
\end{equation}
where  
\begin{align}
&\mbox{$U\in \mathcal S(\R,\R)$ (the Schwartz space),}\label{eq:aa}\\
&\mbox{$p\in\mathcal C^\infty(\R,\R)$ and $\exists C>0$ /  $\forall\phi\in [-1,1]$,
$|\pp(\phi)|\leq C|\phi|^{4}$.}\label{eq:pp}
\end{align}
We call \emph{breather} a time-periodic solution $\phi$ of \eqref{eq:nlkg} 
such that $(\phi,\partial_t \phi)\in \mathcal C(\R, H^1(\R)\times L^2(\R))$.

A well-known example is the family of breathers of the sine-Gordon equation \cite{DaPe,Lamb}
\begin{equation}\label{eq:sG}
\partial_t^2 \phi = \partial_x^2 \phi - \sin \phi
\end{equation}
explicitly given by
\begin{equation}\label{eq:BR}
\mathcal B_{\eps}  (t,x)= 4\arctan\left(\frac{\eps}{\omega} \frac{\cos (\omega t)}{\cosh(\eps x)}\right)
\quad \mbox{for $\eps,\omega\in(0,1)$ with $\eps^2+\omega^2=1$.}
\end{equation}
By the invariance of the sine-Gordon equation by time and space translations, for all $\theta\in [0,\frac{2\pi}\omega)$
and all $r\in \R$, $(t,x)\mapsto \mathcal{B}(t-\theta,x-r)$ is also a solution of \eqref{eq:sG}.

The existence \emph{versus} nonexistence of breathers for general, nonintegrable,
nonlinear Klein-Gordon equations is a classical question which has attracted a lot of attention.
A pioneering work \cite{sk} studied by Fourier series expansion the nonexistence of small breathers 
for the $\phi^4$ equation linearized around the solution $\phi(t,x)\equiv 1$.
Subsequent results of different nature \cite{bmckw, denzler, denzler2, kiche,rainer} proved that, in a certain sense,
the family of breathers does not globally persist for suitable, yet general semilinear perturbations of the sine-Gordon equation.
We also refer the interested reader to the discussion in \cite[pp. 18-19]{SoWe}.

In a more recent article \cite{ggsz}, using a dynamical system approach,
the nonexistence of small, ``single-bump'' breathers is proved for a large class of one-dimensional nonlinear
Klein-Gordon equations with generic analytic odd nonlinearities.
The genericity assumption is related to the nonvanishing of a constant (called there the Stokes constant), 
which depends analytically on the nonlinearity.
As in~\cite{sk}, it is clearly pointed out in \cite{ggsz} that for general Klein-Gordon
equations that are perturbative of the sine-Gordon equation \eqref{eq:sG},
the obstruction to the existence of breathers of size
$0<\varepsilon\ll 1$ is exponentially small in $\varepsilon$. In particular, in the case
where the Stokes constant is zero, ``generalized breathers'' with tails of size exponentially small
in $\varepsilon$, are constructed as an illustration of the subtlety of the nonexistence problem.
We refer to \cite{grsch 1,grsch 2,grsch 3,lu} for other constructions of generalized breathers.

In the present article, we do not address directly the nonexistence question, 
but instead we study the \emph{small energy limit} of a supposedly existing sequence of breathers of \eqref{eq:nlkg}.
To state and motivate the main result, we define the function $Q$ by
\begin{equation}\label{eq:Qy}
Q(y) = \frac 4{\cosh(y)}
\quad\mbox{satisfying}\quad
Q''-   Q+\frac{Q^3}{8} =0 \quad \mbox{on $\R$.}
\end{equation}
We also observe that the family of sine-Gordon breathers has the following small amplitude asymptotic behavior
\begin{equation}\label{eq:Be}
\mathcal B_\eps (t,x) \sim \eps Q(\eps x) \cos (\omega t)\quad \mbox{as $\varepsilon\to 0^+$, with $\eps^2+\omega^2=1$.}
\end{equation}

Our main result shows a rigidity of the breather behavior
in the small energy limit for the general model \eqref{eq:nlkg}.

\begin{theorem}\label{thm:1}
Assume \eqref{eq:aa}-\eqref{eq:pp}.
Let $\{\phi_k\}_{k}$ be a sequence of breathers of~\eqref{eq:nlkg}.
Suppose
\begin{enumerate}
\item[\textnormal{(H1)}]\emph{Periodicity.}
Each $\phi_k$ is $T_k$-periodic with $\sup_k T_k<+\infty$.
\item[\textnormal{(H2)}]\emph{Small energy.}
The sequence $\{\alpha_k\}_k$ defined by 
\[
\alpha_k  = \sup_{t\in\R} \|\phi_k(t)\|_{H^1(\R)}^2\quad \mbox{satisfies}\quad \lim_{k\to +\infty} \alpha_k=0.
\]
\item[\textnormal{(H3)}] \emph{Amplitude-energy relation.}
There exists a constant $c>0$ such that 
\[
\frac1{T_k}\int_0^{T_k}\|\phi_k(t)\|_{L^\infty(\R)}^2 \ud t\geq c \alpha_k^2.
\]
\end{enumerate}
Then, up to the extraction of a subsequence, still denoted by $\{\phi_k\}_k$, it holds
\begin{enumerate}[label=\textnormal{(\roman*)}]
\item\emph{Asymptotic behavior of the period.}\label{th:i}
There exist a positive integer $n_\star$ and a positive real $\lambda>0$ 
such that, as $k\to +\infty$,
\[
\frac{T_k}{2\pi n_\star} = 1+\frac\lambda 2 \alpha_k^2 + o(\alpha_k^2).
\]
\item\emph{Decomposition into a sum of sine-Gordon breathers.}\label{th:ii}
Let
\[
\omega_k= \frac{2\pi n_\star}{T_k}\quad\mbox{and}\quad
\eps_k=\sqrt{1-\omega_k^2}.
\]
There exists an integer $J\geq 1$ and for all $j\in \{1,\ldots, J\}$,
there exist $\theta_j\in [0,2\pi)$ and a sequence of
reals $\{r_{j,k}\}_k$ such that
\[
\lim_{k\to+\infty} \eps_k^{-1} 
\sup_{t\in \R} \int 
\biggl|\phi_k(t,x)-\eps_k \sum_{j=1}^J 
Q(\eps_k x -r_{j,k}) \cos(\omega_k t - \theta_j)\biggr|^2 \ud x =0.
\]
Moreover, if $J\geq 2$, for any $i,j\in \{1,\ldots,J\}$, $i\neq j$,
$\lim_{k\to +\infty} |r_{i,k}-r_{j,k}|=+\infty$.
\item \emph{Fermi golden rule.}\label{th:iii}
If $\widehat \UU(\sqrt{3})\neq 0$ then 
for all $j\in \{1,\ldots,J\}$, $\lim_{k\to+\infty} |r_{j,k}|=+\infty$.\\
In particular, for any $L>0$,
\[
\lim_{k\to+\infty}
\eps_k^{-1} \sup_{t\in\R}\int_{|x|<L\eps_k^{-1}} |\phi_k(t,x)|^2 \ud x   =0.
\]
\end{enumerate}
\end{theorem}

\begin{remark}
The asymptotic result in \ref{th:i} of Theorem \ref{thm:1} implies that
$\omega_k^2 = 1-\lambda \alpha_k^2 + o(\alpha_k^2)$ and
$\varepsilon_k\sim \sqrt{\lambda} \alpha_k$ as $k\to+\infty$.
Moreover, \ref{th:ii} of Theorem \ref{thm:1} can be rewritten as
\begin{equation}\label{eq:BB}
\lim_{k\to+\infty} \eps_k^{-1} 
\sup_{t\in \R} \int \biggl|\phi_k(t,x)-\sum_{j=1}^J  \mathcal{B}_{\varepsilon_k}(t-\tau_{j},x-m_{j,k})\biggr|^2 \ud x  
=0
\end{equation}
where $\tau_j=\omega_k^{-1} \theta_j$ and $m_{j,k}=\varepsilon_k^{-1} r_{j,k}$.
Thus, \ref{th:i} and \ref{th:ii} of Theorem \ref{thm:1} mean that independently of their periods $\{T_k\}_k$, 
small breathers $\{\phi_k\}_k$ are asymptotically close in 
the rescaled $L^2$ norm to a finite sum of small, decoupled,
sine-Gordon breathers, all with the same amplitude $\varepsilon_k\sim \sqrt{\lambda} \alpha_k$
and the same period $2\pi/\omega_k\sim 2\pi$. 
Concerning the rescaled $L^2$ norm, we observe by change of variable that
\[
\varepsilon^{-1} \int |\mathcal{B}_\varepsilon(t,x)|^2 \ud x
= \cos^2(\omega t)  \int Q^2,
\]
which means that \ref{th:ii} of Theorem \ref{thm:1} actually describes the leading order 
asymptotic behavior of the sequence $\{\phi_k\}_k$.
\end{remark}

\begin{remark}
The separation statement concerning the sequences of positions $(r_{j,k})_k$,
in the case $J\geq 2$, is consistent with the particular solutions of the sine-Gordon model
obtained by integrability.
We refer to \cite{Hirota,Schiebold} for the justification for the observations below.
By the integrability theory, one can construct parallel breathers
for the sine-Gordon equation, but then either the breathers do not have the same period (and thus the
parallel $2$-breather is not periodic in time), or at least one breather is not small (and so, 
the periodic solution obtained does not enter the range of applicability of Theorem \ref{thm:1}).
In particular, there does not exist
parallel $2$-breather with equal amplitude. Indeed, $2$-breather with equal 
amplitude rather form \emph{clusters}, \emph{i.e.} they separate logarithmically in time.
Such clusters are not periodic in time.
\end{remark}

\begin{remark}
Assumption (H3) is natural since it is satisfied by the small sine-Gordon breathers.
Indeed, for $\varepsilon$ small,
\[
\sup_{t\in\R}\|\mathcal{B}_\varepsilon(t)\|_{H^1} \approx C \sqrt{\varepsilon} , \quad
\sup_{t\in\R}\|\mathcal{B}_\varepsilon(t)\|_{L^\infty} \approx 4 \varepsilon.
\]
This means that for small breathers  the amplitude is related to the size in $L^2$ in a suitable sense.
The proof of the rigidity result strongly relies on this assumption.
See also Remark \ref{rk:co}.
\end{remark}

\begin{remark}
The sine-Gordon equation also has a standing kink solution
\[
K(x) = 4 \arctan (e^x).
\]
There exists another family of periodic solutions, called \emph{wobbling kinks},
which are in a certain sense superposition of the kink and a breather (see \emph{e.g.} \cite{cqs,Lamb}).
By integrability theory, it is possible to prove that generically, for all suitable initial data, the solution decomposes
in large time as a sum of such kinks, breathers and wobbling kinks, see \cite{Gong}.

Moreover, the existence of breathers and wobbling kinks for the sine-Gordon equation
shows that the asymptotic stability of the zero solution and
of the kink is not true for general initial data in the energy space.
See however, results of asymptotic stability of the kink under symmetry assumptions in \cite{amp}
and \cite{ls}.

Linearizing the sine-Gordon equation around the kink, \emph{i.e.} setting
$\psi(t,x)=K(x)+\phi(t,x)$ leads to
\[
\partial_t^2 \phi = \partial_x^2\phi - \phi + 2 \sech^2(x) \phi 
-\sech(x)\tanh(x)\phi^2 + \left(\frac 16 -\frac 13 \sech^2(x)\right) u^3 + p(x,\phi),
\]
where $p(x,\phi)$ is of order higher than cubic in $\phi$.
We observe first that the linearized operator is resonant
(the function $x\mapsto \tanh(x)$ is the generalized eigenfunction corresponding to
 the bottom of the continuous spectrum).
Second, we see that the potential $U(x)=-\sech(x)\tanh(x)$ in front of the
quadratic term is in the Schwartz space, but satisfies the resonant condition
$\widehat U(\sqrt{3})=0$.
These two difficulties related to wobbling kinks are
addressed in \cite{lls,llss} using Fourier analysis methods and weighted Sobolev spaces.
\end{remark}

\begin{remark}
The structure of the linearization of the sine-Gordon equation around the kink 
seen in the previous remark  led us to 
add the quadratic term $U(x) \phi^2$ with a localized potential in the nonlinear Klein-Gordon
model \eqref{eq:nlkg}. 
Another motivation
comes, for example, from the linearization of the cubic Klein-Gordon equation 
\[
\partial_t^2\psi = \partial_x^2 \psi - \psi  +\frac16 \psi^3
\]
around the soliton solution $\psi(t,x)=\frac{2}{\sqrt{3}} Q(x)$.
Setting $\psi(t,x)=\frac{2}{\sqrt{3}} Q(x)+\phi(t,x)$, we obtain
\[
\partial_t^2\phi = \partial_x^2\phi -\phi + \frac 23 Q^2\phi +  \frac{1}{\sqrt{3}} Q \phi^2+\frac16 \phi^3 .
\]
As for the sine-Gordon equation, the presence of a resonance for the linear operator above is a 
serious spectral difficulty. However, unlike for the sine-Gordon equation, 
the potential $\frac{1}{\sqrt{3}}Q$ in front of the quadratic term satisfies the non resonance condition $\frac{1}{\sqrt{3}}\widehat Q(\sqrt{3})\neq 0$.
This is expected to be connected to the nonexistence of breathers close to the soliton
and to the more difficult question of asymptotic stability of the soliton
(see also \cite{KMM3}).

In the general model \eqref{eq:nlkg}, 
we have not included a linear potential to avoid unnecessary spectral issues, such as exponential instability for example. We believe that \eqref{eq:nlkg}
is an interesting model case
that includes the fundamental difficulty of having a resonance at the bottom
of the continuous spectrum (for \eqref{eq:nlkg}, it is just the constant function).
The phenomenon that we wanted to illustrate in \ref{th:iii} of Theorem \ref{thm:1} is how the (generic) nonvanishing 
condition $\widehat \UU(\sqrt{3})\neq 0$
prevents the presence of a small breather located in a bounded set around zero.
\end{remark}

\begin{remark}
A similar question concerns the model
\begin{equation}\label{eq:p4}
\partial_t^2\psi=\partial_x^2\psi+\psi- \psi^3
\end{equation}
known as the $\phi^4$ equation.
The nonexistence of arbitrarily small breathers for this model is a simple consequence of the 
defocusing sign of the
nonlinearity, see \cite[Theorem 15]{KMM 5}.
However, the existence or nonexistence of wobbling kinks for \eqref{eq:p4} 
in the neighborhood of the kink
it is a notoriously difficult open question. See for example \cite{box} for a formal approach.

Recall that the kink for this model is given by
\[
H(x)=\tanh\left(\frac{x}{\sqrt{2}}\right).
\] 
Writing $\psi=H+\phi$ leads to the following problem
\begin{equation}\label{eq:fi4lin}
\partial_t^2 \phi = \partial_x^2\phi -2\phi + \frac3{\cosh^2(x/\sqrt{2})} \phi 
-3H(x) \phi^2- \phi^3.
\end{equation}
Note that the linearized operator also has a resonance. Moreover, the potential in front of
the quadratic term is not localized, which is an additional difficult of this problem.
\end{remark}

\begin{remark}
We mention a few recent works related to the dynamics around  kink or solitons
in one dimensional models \cite{ KM 1, KMM, KMM 2, KMM3} and for the sine-Gordon equation \cite{amp,gong 2, ls}. 

Other works studied the asymptotic stability of the zero solution for equations of the type
\[
\partial_t^2 \phi =\partial_x^2 \phi-V(x)\phi+\UU(x) \phi^2+\phi^3,
\]
see \cite{GP,lls,llss}. General picture is that $\phi=0$ is asymptotically stable
(and modified scattering holds) in a weighted Sobolev space chosen in such a way that small breathers are not admissible perturbations.
\end{remark}

\subsection{Sketch of the proof}
There are two main steps in the proof of Theorem \ref{thm:1}. 

First, we show that as $k\to +\infty$, the main order of the energy of the breather is concentrated at one  Fourier mode, called the dominant mode.
This first step is partly inspired by \cite{sk} and \cite{ggsz}, where the identification of the
dominant mode and the emergence of the canonical soliton $Q$ appear.

Second, we use the equation of the dominant mode, together with a standard compactness result
(\cite[Lemma 3.1]{BeCe}, \cite[Thm. III.4]{lions}),
to show that this mode decomposes into a sum of decoupled soliton profiles.
We refer to \cite{Fe} for a similar use of a compactness argument
in the context of the Klein-Gordon equation with damping.

\subsection{Notation}
% In what follows, we will denote 
% \[
% \|h\|_{L_T^2 L_x^2}=\left(\int_0^T \int_{\R} |h(t,x)|^2\ud x \ud t\right)^{1/2},\quad
% \|h\|_{L_T^2 L_x^\infty}=\left(\int_0^T \sup_{x\in \R} |h(t,x)|^2 \ud t\right)^{1/2}.
% \]
For convenience, we shall sometimes rewrite equation \eqref{eq:nlkg} as
\[
\partial_t^2\phi=\partial_x^2\phi -\phi + q(x,\phi),
\quad (t,x)\in \R\times\R,
\]
where
\[
q(x,\phi)= \UU(x) \phi^2+\frac 16\phi^3+\pp(\phi).
\]
We denote
\[
  \PP(\phi)=\int_0^\phi p.
\]

\section{Energy bounds}

In this section, we prove general bounds on a given solution $\phi\in \mathcal C(\R,H^1(\R))
\cap \mathcal C^1(\R,L^2(\R))$ of \eqref{eq:nlkg} which is supposed to be $T$-periodic in time
and such that
\[
\alpha = \sup_{t\in\R} \|\phi(t)\|_{H^1(\R)}^2
\]
is sufficiently small.

\begin{lemma}\label{le:bd}
If $\alpha>0$ is sufficiently small, then
\[
\frac 1T \int_0^T \left( \|\partial_t\phi\|_{L^2}^2+\|\phi\|_{L^2}^2\right)\lesssim  \alpha,\quad
\frac 1T \int_0^T \|\partial_x \phi\|_{L^2}^2 \lesssim  \alpha^3,\quad
\frac 1T \int_0^T \|\phi\|_{L^\infty}^4 \lesssim   \alpha^4.
\]
Moreover,
\[
\sup_{t\in\R} \|\partial_t \phi(t)\|_{L^2}^2\lesssim \alpha,\quad 
\frac 1T \int_0^T \|\phi\|_{L^2}^2 \gtrsim \alpha.
\]
\end{lemma}
\begin{proof}
Before starting, note that
\[
\|\phi\|_{L^\infty(\R)}^2\lesssim \|\phi\|_{L^2(\R)}\|\partial_x\phi\|_{L^2(\R)}
\lesssim \alpha.
\]
In particular, for $\alpha>0$ sufficiently small, in view of \eqref{eq:pp}, we have
\[
|\pp(\phi)|\lesssim  |\phi|^4\lesssim  |\phi|^3,\quad
|\PP(\phi)|\lesssim  \phi^5\lesssim \phi^4.
\]

First, define
\[
\mathcal I_1 = \int  \phi \partial_t \phi.
\]
Then, using \eqref{eq:nlkg} and integrating by parts,
\begin{align*}
\mathcal I_1' & = \int  (\partial_t \phi)^2
- \int  ((\partial_x \phi)^2+\phi^2)
 +\int  \left(\UU(x)\phi^3
+\frac 16  \phi^4 +  \phi \pp(\phi)\right).
\end{align*}
Integrating on $[0,T]$, using $\mathcal I_1(0)=\mathcal I_1(T)$, we obtain
\[
\int_0^T 
\int  (\partial_t \phi)^2
- \int_0^T \int  ((\partial_x \phi)^2+\phi^2)
+\int_0^T \int  \left(\UU(x)\phi^3
+\frac 16  \phi^4 +  \phi \pp(\phi)\right)= 0 .
\]
Moreover,
\[
\int \left(|\UU(x)\phi^3|+\phi^4 + |\phi \pp(\phi)|\right)
\lesssim \|\phi\|_{L^\infty(\R)}^3+ 
\|\phi\|_{L^\infty(\R)}^2 \|\phi\|_{L^2(\R)}^2
\lesssim \alpha^\frac32.
\]
Thus, by the definition of $\alpha$,
\begin{equation}\label{eq:di}
\frac 1T \int_0^T \int  (\partial_t \phi)^2 \lesssim 
\frac 1T\int_0^T \int  ((\partial_x \phi)^2+\phi^2) 
+ \alpha^\frac32\lesssim \alpha.
\end{equation}

Second, for $A>1$, define
\[
\zeta_A(x)=\frac 1{\cosh\left(\frac xA \right)},\quad
\Theta_A(x)=\int_0^x \zeta_A.
\]
Note that $|\Theta_A(x)|\leq |x|$ and $|\Theta_A(x)|\lesssim A$ on $\R$.
Set
\[
\mathcal I_2 = \int \left(\Theta_A \partial_x \phi + \frac12\zeta_A \phi\right)\partial_t \phi.
\]
Then, using \eqref{eq:nlkg} and integration by parts,
\begin{align*}
\mathcal I_2' &= 
- \int \zeta_A (\partial_x \phi)^2 + \frac 14 \int \zeta_A'' \phi^2
+ \frac16 \int \zeta_A \UU(x) \phi^3
-\frac 13 \int \Theta_A U'(x) \phi^3\\
&\quad + \int \zeta_A \left( \frac 1{24}\phi^4 - \PP(\phi)+ \frac12 \phi \pp(\phi)\right).
\end{align*}
Integrating on $[0,T]$, using $\mathcal I_2(0)=\mathcal I_2(T)$, we obtain
\begin{align*}
\int_0^T  \int \zeta_A (\partial_x \phi)^2 &= \frac 14 \int_0^T\int \zeta_A'' \phi^2
+\frac16 \int_0^T \int \zeta_A \UU(x) \phi^3
-\frac 13 \int_0^T\int \Theta_A \UU'(x) \phi^3\\
&\quad + \int_0^T\int \zeta_A \left( \frac 1{24}\phi^4 - \PP(\phi)+ \frac12 \phi \pp(\phi)\right)
\end{align*}
Passing to the limit as $A\to +\infty$, we find the identity
\begin{align*}
\int_0^T  \int  (\partial_x \phi)^2 
=\frac16 \int_0^T \int  \UU(x) \phi^3
-\frac 13 \int_0^T\int x U'(x) \phi^3
+ \int_0^T\int \left( \frac 1{24}\phi^4 - \PP(\phi)+ \frac12 \phi \pp(\phi)\right)
\end{align*}
Thus, estimating the nonlinear terms as before and using the Hölder inequality in time,
we obtain
\begin{align*}
\int_0^T\int (\partial_x \phi)^2
& \lesssim \int_0^T \|\phi(t)\|_{L^\infty}^3 \ud t
+\int_0^T \|\phi(t)\|_{L^\infty}^2 \|\phi(t)\|_{L^2}^2 \ud t\\
& \lesssim \left( \int_0^T \int (\partial_x \phi)^2 \right)^\frac34
\left(\int_0^T\left(\int\phi^2\right)^3\right)^\frac14
+ \left( \int_0^T \int (\partial_x \phi)^2\right)^\frac12
\left(\int_0^T\left(\int\phi^2\right)^3\right)^\frac12.
\end{align*}
This implies
\begin{equation}\label{eq:dd}
\int_0^T\int (\partial_x \phi)^2
\lesssim  \int_0^T\left(\int\phi^2\right)^3 
\lesssim T \alpha^3.
\end{equation}
In particular, we also have
\[
\int_0^T \|\phi\|_{L^\infty}^4 
\lesssim \int_0^T \|\partial_x \phi\|_{L^2}^2 \|\phi\|_{L^2}^2
\lesssim \sup_{[0,T]} \|\phi\|_{L^2}^2 \int_0^T \|\partial_x \phi\|_{L^2}^2
\lesssim T \alpha^4.
\]

Lastly, we use the energy conservation. Define
\[
\mathcal E(t) = \int \left( (\partial_t \phi)^2 + (\partial_x \phi)^2 + \phi^2
- \UU(x) \frac{\phi^3}{3} - \frac 1{24} \phi^4 - \PP(\phi) \right)
\]
so that using \eqref{eq:nlkg}, it holds $\mathcal E(t)=\mathcal E(0)$ for all $t\in \R$.
As before, we have the estimate
\[
\int |\UU(x)| |\phi|^3+\phi^4+|\PP(\phi)|\lesssim \alpha^\frac32.
\]
But for any $t\in\R$, by \eqref{eq:di}
\[
 \mathcal E(t) = \frac 1T\int_0^T\mathcal E \lesssim \alpha.
\]
Thus, $\mathcal E(t) \lesssim \alpha$, which implies by the definition of 
$\alpha$ that for all $t\in\R$,
\[
\int  (\partial_t\phi)^2  \lesssim \alpha.
\]

Moreover, we also get by the definition of $\alpha$ that, for any $t\in\R$,
\begin{equation}\label{eq:d1}
\mathcal E(t) \gtrsim \alpha.
\end{equation}
By \eqref{eq:di} and \eqref{eq:dd}, we have
\[
\int_0^T \int  (\partial_t \phi)^2 \lesssim \int_0^T \int  ((\partial_x \phi)^2+\phi^2) 
+T \alpha^\frac32\lesssim \int_0^T \int   \phi^2 + T \alpha^\frac32.
\]
Thus, by \eqref{eq:d1} 
\[
\frac 1T\int_0^T \int  \phi^2 
\geq  \frac CT\int_0^T \int \left(\phi^2 + (\partial_t\phi)^2\right) -C\alpha^\frac32
\gtrsim \mathcal E(t) \gtrsim \alpha.
\]

\end{proof}

\section{Proof of Theorem \ref{thm:1}}

Let $\{\phi_k\}_k$ be a sequence of solutions of \eqref{eq:nlkg} as in the statement of Theorem \ref{thm:1}.

\subsection{Rescaled variables}
% Up to a subsequence, there are two possible limits for the bounded sequence $\{T_k\}_k$
% of positive reals
% \[
% \limsup_{\eps\to 0}T_\eps=
% \begin{cases}
% +\infty\\
% 0<T_\star<+\infty
% \\ 0
% \end{cases}
% \]
% We start with the first two case and take a subsequence such that $\lim_{k\to +\infty} T_k=T_\star>0$. 

We introduce rescaled time-space variables $(s,y)$ setting
\[
v_k(s,y)= \frac 1{\alpha_k} \phi_k\left(\frac{T_k}{2\pi} s, \frac{y}{\alpha_k}\right).
\]
Equivalently,
\[
\phi_k(t,x)=\alpha_k v_k\left(\frac{2\pi}{T_k} t,\alpha_k x\right).
\]
From \eqref{eq:nlkg}, $v_k(s,y)$ satisfies
\begin{equation}\label{eq:vv}
\frac 1{\alpha_k^2} \frac{4\pi^2}{T_k^2}\partial_s^2 v_k - \partial_y^2 v_k
+\frac1{\alpha_k^2} v_k - \qqq_k =0
\end{equation}
where
\begin{equation}\label{eq:fk}
  \qqq_k =\Uk  v_k^2
  +\frac16 v_k^3 + \frac1{\alpha_k^3}\pp(\alpha_k v_k) \quad \mbox{where} \quad
  \Uk(y)=\frac1{\alpha_k} U\left(\frac y{\alpha_k} \right).
\end{equation}
Moreover, each function $v_k$ is $2\pi$-periodic in time.

Denote, for $p,q\in[1,\infty)$,
\begin{align*}
  &\|v\|_{L^p_sL^q_y}= \left(\int_0^{2\pi}\left(\int_\R |v(s,y)|^q \ud y\right)^\frac{p}q \ud s\right)^\frac1p,
  \quad \|v\|_{L^\infty_sL^q_y}=\sup_{s\in[0,2\pi]} \left(\int_\R |v(s,y)|^q \ud y\right)^\frac1q,\\
  &\|v\|_{L^p_sL^\infty_y}=\left(\int_0^{2\pi}\left(\sup_{y\in\R}|v(s,y)|^p\right)\ud s\right)^\frac1p, \quad \|v\|_{L^\infty_sL^\infty_y}=\sup_{s\in [0,2\pi]} \sup_{y\in \R} \left|v(s,y)\right|.
\end{align*}
We rewrite the results of Lemma \ref{le:bd} and (H3) in terms of the function $v_k$.
\begin{lemma}[Estimates on the rescaled sequence]\label{le:bv}
\begin{enumerate}
\item By Lemma \ref{le:bd},
  \begin{align*}
  & \|v_k\|_{L^\infty_sL^2_y} +\alpha_k \|\partial_y v_k\|_{L^\infty_sL^2_y} 
  + T_k^{-1}\|\partial_s v_k\|_{L^\infty_sL^2_y}+\|v_k\|_{L_s^4L^\infty_y} \lesssim 1,\\
  & \|v_k\|_{L^2_sL^2_y}+\|\partial_yv_k\|_{L^2_sL^2_y}+\|v_k\|_{L^2_sL^\infty_y}
  +T_{k}^{-\frac12}\|\partial_sv_k\|_{L^2_sL^2_y}\lesssim 1,\\
  & \|v_k\|_{L^2_sL^2_y} \gtrsim 1.
\end{align*}
\item Assuming {\rm (H3)}, it holds
\[
 \|v_k\|_{L^2_sL^\infty_y}\gtrsim 1.
\]
\end{enumerate}
\end{lemma}
In particular, the implicit constants are independent of $T_k$.

\subsection{Emergence of a dominant mode}
Since $s\mapsto v_k(s,y)$ is $2\pi$-periodic in time, we decompose it in Fourier series as follows
\[
v_k(s,y)=\frac{1}{2} a_{k,0}(y)+\sum_{n=1}^\infty a_{k,n}(y)\cos( n s )
+\sum_{n=1}^\infty b_{k,n}(y)\sin( n s )
\]
where
\begin{align*}
  & a_{k,n}(y)=\frac1\pi\int_0^{2\pi}v_k(s,y)\cos(ns)\ud s, \mbox{ for $n\geq 0$,} \\
  & b_{k,n}(y)=\frac1\pi\int_0^{2\pi}v_k(s,y)\sin(ns)\ud s,\mbox{ for $n\geq 1$.}
\end{align*}
Similarly, set
\[
q_k(s,y)=\frac{1}{2} f_{k,0}(y)+\sum_{n=1}^\infty f_{k,n}(y)\cos( n s )
+\sum_{n=1}^\infty g_{k,n}(y)\sin( n s )
\]
where
\begin{align*}
  & f_{k,n}(y)=\frac1\pi\int_0^{2\pi}q_k(s,y)\cos(ns)\ud s, \mbox{ for $n\geq 0$,}\\
  & g_{k,n}(y)=\frac1\pi\int_0^{2\pi}q_k(s,y)\sin(ns)\ud s,\mbox{ for $n\geq 1$.}
\end{align*}
From \eqref{eq:vv}, it follows that
\begin{equation}\label{eq:vn}
  \begin{split}
    & -a_{k,n}'' + \lambda_{k,n} a_{k,n} - f_{k,n} = 0, \mbox{ for $n\geq 0$,}\\
    & -b_{k,n}'' + \lambda_{k,n} b_{k,n} - g_{k,n} = 0,\mbox{ for $n\geq 1$,}
  \end{split}
\end{equation}
where for $n\geq 0$,
\begin{equation*}
  \mu_{k,n}= \frac {4\pi^2n^2}{T_k^2} - 1, \quad \lambda_{k,n}=-\frac{\mu_{k,n}}{\alpha_k^2}.
\end{equation*}

\begin{lemma}\label{le:vp}
There exists an integer $n_\star\geq 1$ such that, up to a subsequence,
\begin{equation*}
  \lim_{k\to +\infty} T_k =2\pi n_\star.
\end{equation*}
Moreover, decomposing
\begin{equation*}
  v_k = v_{k,\star} + v_{k,\perp}
\end{equation*}
where
\begin{align*}
  v_{k,\star}(s,y) & = a_{k,n_\star}(y) \cos( n_\star s)+b_{k,n_\star}(y) \sin( n_\star s),\\
  v_{k,\perp}(s,y) & = \frac 12 a_{k,0}(y)
  + \sum_{n\not\in\{0,n_\star\}} \left[ a_{k,n}(y) \cos(ns) + b_{k,n}(y) \sin(ns) \right],
\end{align*}
it holds
\begin{align*}
  & \|a_{k,n_\star}\|_{L^2}+\|b_{k,n_\star}\|_{L^2}\gtrsim \|v_{k,\star}\|_{L^2_sL^2_y}\gtrsim 1,\\
  & \|a_{k,n_\star}\|_{L^\infty} +\|b_{k,n_\star}\|_{L^\infty}\gtrsim \|v_{k,\star}\|_{L^2_sL^\infty_y}\gtrsim 1,
\end{align*} 
and
\begin{equation*}
  \|\partial_s v_{k,\perp}\|_{L^2_sL^2_y}+
  \|v_{k,\perp}\|_{L^2_sL^2_y}+\|v_{k,\perp}\|_{L^\infty_sL^2_y}  \lesssim \alpha_k,
  \quad \|v_{k,\perp}\|_{L^4_sL^\infty_y} \lesssim \alpha_k^\frac12.
\end{equation*}
\end{lemma}
\begin{proof}
By assumption (H1), up to the extraction of a subsequence, there exists 
\[
T_\star = \lim_{k\to+\infty} T_k\in [0,+\infty).
\]
If $\frac 1{2\pi}T_\star\in \N$ then we assume after the extraction of a subsequence that
$\frac1{2\pi}|T_k-T_\star|<\frac12$, for all $k\geq 1$.
Otherwise, we assume $\frac1{2\pi}|T_k-T_\star|<\frac12\DIST(\frac1{2\pi}T_\star,\N)$.
In both cases, for all $k\geq 1$, for all $n\geq 1$, one has the equivalences
\begin{equation*}
  n<\frac{T_\star}{2\pi} \iff n<\frac{T_k}{2\pi} \iff \mu_{k,n}<0,\quad
  n>\frac{T_\star}{2\pi} \iff n>\frac{T_k}{2\pi} \iff \mu_{k,n}>0.
\end{equation*}
Moreover, there exists $\delta>0$ (independent of $k$ and $n$) such that for all $n\in\N$,
\begin{equation}\label{eq:lb}
  n\neq  \frac{T_\star}{2\pi} \implies |\mu_{k,n}| \geq \delta (1+n^2).
\end{equation}
For the moment, set
\begin{equation*}
  v_{k,\perp}(s,y) = \frac 12 a_{k,0}(y) 
  + \sum_{n\not\in\{0,\frac 1{2\pi}T_\star\}} \left[a_{k,n}(y) \cos(ns) + b_{n,k}(y)\sin(ns)\right]
\end{equation*}
If $T_\star=0$ or if $\frac1{2\pi}{T_\star}\not\in\N$,
the summation on the right-hand side includes all $n\geq 1$.
If $\frac 1{2\pi}T_\star\in \N$, it includes all $n\geq 1$ except $n=\frac 1{2\pi}T_\star$.
Thus, if $T_\star=0$ or if $\frac1{2\pi}{T_\star}\not\in\N$ then one has $v_{k,\perp}=v_k$.
The main point of this proof is to show that $v_{k,\perp}\neq v_k$, which will imply
that $\frac 1{2\pi}T_\star\in \N\setminus\{0\}$.

From the definition of $\qqq_k$, we see that
\begin{equation}\label{eq:bf}
  |\qqq_k|\lesssim |\Uk| v_k^2+ |v_k|^3,
\end{equation}
and so using Lemma \ref{le:bv},
\begin{align*}
  \|\qqq_k \|_{L^2_sL^2_y}^2
  &\lesssim   \int_0^{2\pi}\int  \Uk^2 v_k^4
  + \|v_k^3\|_{L^2_sL^2_y}^2\\
  &\lesssim  \frac1{\alpha_k} \|v_k\|_{L^4_sL^\infty_y}^4 \|\UU\|_{L^2}^2
  + \|v_k\|_{L^4_sL^\infty_y}^4 \|v_k\|_{L^\infty_sL^2_y}^2 \lesssim \frac 1{\alpha_k}.
\end{align*}
Thus,
\begin{equation}\label{eq:bq}
  \|f_{n,k}\|_{L^2}^2+\|g_{n,k}\|_{L^2}^2\lesssim \|\qqq_k \|_{L^2_sL^2_y}^2\lesssim \frac 1{\alpha_k}.
\end{equation}

Let $n=0$.  Multiplying the first line of \eqref{eq:vn} by $a_{k,0}$, integrating on $\R$ and integrating by parts, one finds
\begin{equation*}
  \|a_{k,0}'\|_{L^2}^2+\frac1{\alpha_k^2}\|a_{k,0}\|_{L^2}^2 = \int f_{k,0} a_{k,0}
\end{equation*}
and thus in this case, by the Cauchy-Schwarz inequality,
\begin{equation}\label{eq:R0}
  \|a_{k,0}\|_{L^2}^2 \lesssim {\alpha_k^4}\|f_{k,0}\|_{L^2}^2.
\end{equation}

In the case where $\frac{T_\star}{2\pi}>1$, let $n\in\N$ be such that $1\leq n<\frac{T_\star}{2\pi}$.
Multiplying \eqref{eq:vn} by $a_{k,n}$, integrating on $\R$ and integrating by parts, one finds
\begin{equation*}
  \|a_{k,n}'\|_{L^2}^2+\frac{|\mu_{k,n}|}{\alpha_k^2} \|a_{k,n}\|_{L^2}^2 = \int f_{k,n} a_{k,n}
\end{equation*}
and thus in this case, by \eqref{eq:lb} and  the Cauchy-Schwarz inequality,
\begin{equation}\label{eq:R1}
  \|a_{k,n}\|_{L^2}^2 \lesssim \frac{\alpha_k^4}{(1+n^2)^2} \|f_{k,n}\|_{L^2}^2.
\end{equation}
Similarly, using the second line of \eqref{eq:vn},
\begin{equation}\label{eq:R7}
  \|b_{k,n}\|_{L^2}^2 \lesssim \frac{\alpha_k^4}{(1+n^2)^2} \|g_{k,n}\|_{L^2}^2.
\end{equation}

Let $n\in\N$ be such that $n>\frac{T_\star}{2\pi}$.
Multiplying the first line of \eqref{eq:vn} by $a_{k,n}$, integrating on $\R$ and integrating by parts, one finds
\begin{equation*}
  \frac{|\mu_{k,n}|}{\alpha_k^2} \|a_{k,n}\|_{L^2}^2 = \|a_{k,n}'\|_{L^2}^2 - \int f_{k,n} a_{k,n}
\end{equation*}
and thus in this case, by \eqref{eq:lb} and the Cauchy-Schwarz inequality,
\begin{equation}\label{eq:R2}
  \|a_{k,n}\|_{L^2}^2 \lesssim \frac{\alpha_k^2}{1+n^2}\|a_{k,n}'\|_{L^2}^2
  +  \frac{\alpha_k^4}{(1+n^2)^2}\|f_{k,n}\|^2  .
\end{equation}
Similarly, using the second line of \eqref{eq:vn},
\begin{equation}\label{eq:R8}
  \|b_{k,n}\|_{L^2}^2 \lesssim \frac{\alpha_k^2}{1+n^2}\|b_{k,n}'\|_{L^2}^2
  +  \frac{\alpha_k^4}{(1+n^2)^2}\|g_{k,n}\|^2  .
\end{equation}

Using the definition of $v_\perp$ and the Parceval identity, we deduce from \eqref{eq:R0}-\eqref{eq:R8} that
\begin{equation*}
  \|\partial_s v_{k,\perp}\|_{L^2_sL^2_y}^2+
  \|v_{k,\perp}\|_{L^2_sL^2_y}^2 \lesssim \alpha_k^2 \|\partial_y v_k\|_{L^2_sL^2_y}^2 + \alpha_k^4 \|\qqq_k\|_{L^2_sL^2_y}^2.
\end{equation*}
Therefore, using \eqref{eq:bq} and Lemma \ref{le:bv}, we have proved
\begin{equation*}
  \|v_{k,\perp}\|_{L^2_sL^2_y}  \lesssim \alpha_k  .
\end{equation*}

Since $\|v_k\|_{L^2_sL^2_y}\gtrsim 1$ from Lemma \ref{le:bv}, 
we obtain that $v_{k,\perp}\neq v_k$ and so the existence of a positive integer $n_\star$ such that
$T_\star = 2\pi n_\star$. 
(Note that this argument excludes the case $T_\star=0$.)
As a consequence, we also obtain that $\|a_{k,n_\star}\|_{L^2}+\|b_{k,n_\star}\|_{L^2}\gtrsim \|v_{k,\star}\|_{L^2_sL^2_y} \gtrsim 1$. 
Moreover, the definition of $v_{k,\perp}$ above now matches the one given in the 
statement of the lemma.

Using again \eqref{eq:R1}-\eqref{eq:R8}, we observe that for any $n\neq n_\star$,
\begin{align*}
  \|a_{k,n}\|_{L^2} & \lesssim \frac{\alpha_k}{\sqrt{1+n^2}}
  \left( \|a_{k,n}'\|_{L^2} + \alpha_k \|f_{k,n}\|_{L^2}\right)\\
  & \lesssim \alpha_k \left( \frac 1{1+n^2} + \|a_{k,n}'\|_{L^2}^2 + \alpha_k^2 \|f_{k,n}\|_{L^2}^2\right),
\end{align*}
and a similar estimate for $\|b_{k,n}\|_{L^2}$.
Using the Parceval identity and then Lemma \ref{le:bv} and \eqref{eq:bq}, we deduce that
\begin{equation*}
  \sum_{n\neq n_\star} \left(\|a_{k,n}\|_{L^2} + \|b_{k,n}\|_{L^2}\right)
  \lesssim \alpha_k \left(1+ \|\partial_y v_k\|_{L^2_sL^2_y}^2 + \alpha_k^2 \|\qqq_k\|_{L^2_sL^2_y}^2\right)
  \lesssim \alpha_k.
\end{equation*}
By the triangle inequality, for any $s\in[0,2\pi]$,
\begin{equation*}
  \|v_{k,\perp}(s)\|_{L^2} \lesssim \sum_{n\neq n_\star} \left(\|a_{k,n}\|_{L^2} + \|b_{k,n}\|_{L^2}\right)\lesssim \alpha_k.
\end{equation*}
This proves the desired estimate for $\|v_{k,\perp}\|_{L^\infty_sL^2_y}$.
As a consequence, we also obtain
\begin{equation*}
  \|v_{k,\perp}\|_{L^4_sL^\infty_y}^2\lesssim
  \|v_{k,\perp}\|_{L^\infty_sL^2_y}\|\partial_y v_{k,\perp}\|_{L^2_sL^2_y} 
  \lesssim \alpha_k \|\partial_y v_k\|_{L^2_sL^2_y} \lesssim \alpha_k .
\end{equation*}
In particular, $\|v_{k,\perp}\|_{L^2_sL^\infty_y}\lesssim \alpha_k^\frac12$.
Since $\|v_k\|_{L^2_sL^\infty_y}\gtrsim 1$ from Lemma \ref{le:bv} 
this implies that
\[
\|a_{k,n_\star}\|_{L^\infty} +\|b_{k,n_\star}\|_{L^\infty} \gtrsim \|v_{k,\star}\|_{L^2_sL^\infty_y} 
\gtrsim \|v_k\|_{L^2_sL^\infty_y} \gtrsim 1.
\]
(Recall that $\|v_k\|_{L^2_sL^\infty_y}\gtrsim 1$ was a consequence of the assumption (H3).)
\end{proof}

% 
% \begin{equation*}
%   \|v_{k,\perp}\|_{L^4_sL^4_y}\lesssim
%   \|v_{k,\perp}\|_{L^\infty_sL^2_y}^\frac 34\|\partial_y v_{k,\perp}\|_{L^2_sL^2_y}^\frac14
%   \lesssim \alpha_k^\frac34 \|\partial_y v_k\|_{L^2_sL^2_y}^\frac14 \lesssim \alpha_k^\frac12.
% \end{equation*}
% Lastly, we go back to \eqref{eq:R1} with $n=0$
% \begin{equation*}
%   \|v_{k,0}\|_{L^2} \lesssim \alpha_k^2 \|f_{k,0}\|_{L^2} \lesssim \alpha_k^\frac32.
% \end{equation*}

Now, we prove another estimate based on the method of variation of the constant.

\begin{lemma}\label{le:vc} 
For all $n\neq n_\star$, it holds
\begin{equation*}
  \|a_{k,n}\|_{L^\infty} +\|b_{k,n}\|_{L^\infty} 
\lesssim \alpha_k .
\end{equation*}
\end{lemma}
\begin{proof}
To begin with, we observe 
by the Fubini inequality and then \eqref{eq:bf}, for all $n\geq 0$,
\begin{align*}
  \|f_{k,n}\|_{L^1} +\|g_{k,n}\|_{L^1} \lesssim \|q_{k}\|_{L^1_sL^1_y}
  &\lesssim  \int_0^{2\pi} \int \left[|\Uk| v_k^2
  + |v_k|^3\right]\\
  &\lesssim \|v_k\|_{L^2_sL^\infty_y}^2 \|\UU\|_{L^1} + \|v_k\|_{L^\infty_sL^2_y}^2 \|v_k\|_{L^1_sL^\infty_y}.
\end{align*}
Thus, from Lemma \ref{le:bv},
\begin{equation}\label{eq:if}
  \|f_{k,n}\|_{L^1}+\|g_{k,n}\|_{L^1}\lesssim 1.
\end{equation}
Let $0\leq n < n_\star$. From the proof of Lemma \ref{le:vp} (see \eqref{eq:lb}), we have
$\mu_{k,n}<0$ and 
\begin{equation*}
  \lambda_{k,n} \geq \delta \,\frac{1+n^2}{\alpha_k^2}.
\end{equation*}
Using the fact that $a_{k,n}$ is an $H^1$ solution of the first line of \eqref{eq:vn} on $\R$, we have
\begin{equation*}
  a_{k,n}(y) =
  -\frac{e^{-\sqrt{\lambda_{k,n}} y}}{2\sqrt{\lambda_{k,n}}} \int_{-\infty}^y e^{\sqrt{\lambda_{k,n}} z} f_{k,n}(z) dz
  -\frac{e^{\sqrt{\lambda_{k,n}} y}}{2\sqrt{\lambda_{k,n}}} \int_y^{\infty} e^{-\sqrt{\lambda_{k,n}} z} f_{k,n}(z) dz.
\end{equation*}
Thus, by \eqref{eq:if}, for any $0\leq n< n_\star$, for all $y\in \R$,
\begin{equation*}
  |a_{k,n}(y)|^2 \lesssim \frac 1{\lambda_{k,n}} \|f_{k,n}\|_{L^1}^2\lesssim \frac{\alpha_k^2}{1+n^2} .
\end{equation*}

Let $n>n_\star$. We have $\mu_{k,n}>0$ and 
\begin{equation*}
  - \lambda_{k,n} \geq \delta \,\frac{1+n^2}{\alpha_k^2}.
\end{equation*}
As before, we have
\begin{equation*}
  a_{k,n}(y) =
  -\frac{e^{-\ii\sqrt{-\lambda_{k,n}} y}}{2\ii\sqrt{-\lambda_{k,n}}} \int_{-\infty}^y e^{\ii\sqrt{-\lambda_{k,n}} z} f_{k,n}(z) dz
  +\frac{e^{\ii\sqrt{-\lambda_{k,n}} y}}{2\ii\sqrt{-\lambda_{k,n}}} \int_y^{\infty} e^{-\ii\sqrt{-\lambda_{k,n}} z} f_{k,n}(z) dz,
\end{equation*}
(by \eqref{eq:vn}, we know that $\widehat f_{k,n}(\pm \sqrt{-\lambda_{k,n}})=0$),
and thus, for $n>n_\star$, for all $y\in \R$,
\begin{equation*}
  |a_{k,n}(y)|^2 \lesssim \frac 1{|\lambda_{k,n}|} \|f_{k,n}\|_{L^1}^2\lesssim \frac{\alpha_k^2}{1+n^2} .
\end{equation*}
We proceed similarly to estimate $|b_{k,n}(y)|$ for $n\neq n_\star$.
Therefore, for $n\neq n_\star$,
\begin{equation*}
  \|a_{k,n}\|_{L^\infty}^2+\|b_{k,n}\|_{L^\infty}^2 \lesssim \frac{\alpha_k^2}{1+n^2}
\lesssim \alpha_k^2.
\end{equation*}
\end{proof}

\subsection{Asymptotic behavior of the period}

\begin{lemma}\label{le:bl}
The sequence $(\lambda_{k,n_\star})_k$ is bounded.
\end{lemma}
\begin{proof}
From \eqref{eq:vn}, recall the equations for $(a_{k,n_\star},b_{k,n_\star})$
\begin{equation}\label{eq:vs}
  \begin{aligned}
  & a_{k,n_\star}'' - \lambda_{k,n_\star} a_{k,n_\star} - f_{k,n_\star} =0\\
  & b_{k,n_\star}'' - \lambda_{k,n_\star} b_{k,n_\star} - g_{k,n_\star} =0.
  \end{aligned}
\end{equation}
Multiplying the first equation by $a_{k,n_\star}$, integrating and then integrating by parts, we obtain
\begin{equation*}
  \lambda_{k,n_\star} \|a_{k,n_\star}\|_{L^2}^2 = - \|a_{k,n_\star}'\|_{L^2}^2 - \int f_{k,n_\star} a_{k,n_\star}.
\end{equation*}
Similarly,
\begin{equation*}
  \lambda_{k,n_\star} \|b_{k,n_\star}\|_{L^2}^2 = - \|b_{k,n_\star}'\|_{L^2}^2 - \int g_{k,n_\star} b_{k,n_\star}.
\end{equation*}
By  $\|a_{k,n_\star}\|_{L^2} +\|b_{k,n_\star}\|_{L^2}\gtrsim 1$
(Lemma \ref{le:vp}), we obtain
\begin{equation*}
  |\lambda_{k,n_\star}| \lesssim 
  \|a_{k,n_\star}'\|_{L^2}^2 +\|b_{k,n_\star}'\|_{L^2}^2+ 
  \| a_{k,n_\star}\|_{L^\infty}\| f_{k,n_\star} \|_{L^1}
  +\|b_{k,n_\star}\|_{L^\infty} \| g_{k,n_\star} \|_{L^1} .
\end{equation*}
First, by the Plancherel identity and then Lemma \ref{le:bv},
\begin{equation*}
  \|a_{k,n_\star}'\|_{L^2}^2 +\|b_{k,n_\star}'\|_{L^2}^2 \lesssim \|\partial_y v_k\|_{L^2_sL^2_y}^2
   \lesssim 1.
\end{equation*}
Second, by the Cauchy-Schwarz inequality in the variable $s\in [0,2\pi]$,
and then Lemma \ref{le:bv},
\begin{equation*}
  \|a_{k,n_\star}\|_{L^\infty}+\|b_{k,n_\star}\|_{L^\infty} \lesssim \|v_k\|_{L^1_sL^\infty_y} 
  \lesssim \|v_k\|_{L^2_sL^\infty_y} \lesssim 1.
\end{equation*}
Third, we recall from \eqref{eq:if} that $\|f_{k,n}\|_{L^1}+\|g_{k,n}\|_{L^1}\lesssim 1$.
Therefore, $|\lambda_{k,n_\star}|\lesssim 1$.
\end{proof}
Up to the extraction of a subsequence (as before, we do not change the notation for such extraction), 
there exists $\lambda\in \R$ such that
$\lim_{k\to+\infty} \lambda_{k,n_*} = \lambda$.
This is equivalent to say that
\begin{equation}\label{eq:Tl}
  \lim_{k\to+\infty} \frac 1{\alpha_k^2} \left(1-\frac{T_\star^2}{T_k^2}\right)=\lambda.
\end{equation}
\begin{remark}\label{rk:co}
In the case where $n_\star=1$, it follows from the arguments in \cite{coron}
that $\lambda\geq 0$ without the need of assumption (H3).
In the sequel, we will use assumption (H3) to obtain that $\lambda>0$ 
in the general case.
\end{remark}

\subsection{Equation of the dominant mode}
Our goal is to pass to the limit $k\to +\infty$ in the system \eqref{eq:vs} of $(a_{k,n_\star},b_{k,n_\star})$.
We already know that $(\lambda_{k,n_\star})$ converges to $\lambda$.
Now, we analyze the asymptotic behavior of the second member $(f_{k,n_\star},g_{k,n_\star})$ of \eqref{eq:vs}.

\begin{lemma}\label{le:6}
It holds
\begin{align*}
   \lim_{k\to +\infty} \left\| f_{k,n_\star} - \frac 18 a_{k,n_\star}(a_{k,n_\star}^2+b_{k,n_\star}^2) \right\|_{L^2} & = 0,\\
  \lim_{k\to +\infty} \left\| g_{k,n_\star} - \frac 18 b_{k,n_\star}(a_{k,n_\star}^2+b_{k,n_\star}^2) \right\|_{L^2} & = 0.
\end{align*}
\end{lemma}
\begin{proof}
Recall from \eqref{eq:fk} that
\begin{equation}\label{eq:fk2}
  \qqq_k =\Uk  v_k^2
  +\frac16 v_k^3 + \frac1{\alpha_k^3}\pp(\alpha_k v_k) =\qqq_{k}^{\rm I}+\qqq_{k}^{\rm II}+\qqq_{k}^{\rm III}.
\end{equation}
It is essential for this proof to use that the component $v_{k,\perp}$ of $v_k$ enjoys better estimates
than $v_k$ by Lemma \ref{le:vp}.

First, expanding $v_k = v_{k,\star}+v_{k,\perp}$, one has
\begin{equation*}
  \qqq_k^{\rm I} =\Uk  \left(v_{k,\star}^2+2v_{k,\star}v_{k,\perp}
  +v_{k,\perp}^2\right).
\end{equation*}
By standard trigonometry,
\begin{align*}
  v_{k,\star}^2(s) & = \left(a_{k,n_\star} \cos(n_\star s) + b_{k,n_\star} \sin(n_\star s)\right)^2 \\
  & = a_{k,n_\star}^2 \cos^2(n_\star s) + 2 a_{k,n_\star} b_{k,n_\star} \cos(n_\star s) \sin(n_\star s)
  + b_{k,n_\star}^2 \sin^2(n_\star s)\\
  & = \frac 12 (a_{k,n_\star}^2+b_{k,n_\star}^2)
      + \frac 12 (a_{k,n_\star}^2-b_{k,n_\star}^2) \cos(2 n_*s)
      + a_{k,n_\star} b_{k,n_\star} \sin (2 n_\star s).
\end{align*}
In particular, the first term in $q_k^{\rm I}$
does not contribute to $f_{k,n_\star}^{\rm I}$ nor to $g_{k,n_\star}^{\rm I}$.
Moreover,
\begin{align*}
  &\frac 1\pi\int_0^{2\pi} v_{k,\perp} v_{k,\star} \cos (n_\star s)
   = \frac 1\pi\int_0^{2\pi} v_{k,\perp} \left( a_{k,n_\star} \cos^2(n_\star s) + b_{k,n_\star} \sin(n_\star s)
  \cos(n_\star s)\right)\\
  &\quad = \frac {a_{k,n_\star}}{2\pi} \int_0^{2\pi} v_{k,\perp} 
  + \frac {a_{k,n_\star}}{2\pi} \int_0^{2\pi} v_{k,\perp}  \cos(2n_\star s)
  + \frac {b_{k,n_\star}}{2\pi} \int_0^{2\pi} v_{k,\perp}  \sin(2n_\star s)\\
  &\quad = \frac 12 a_{k,0}a_{k,n_\star} + \frac 12a_{k,n_\star} a_{k,2n_\star}
  +\frac 12 b_{k,n_\star} b_{k,2n_\star},
\end{align*}
and similarly,
\begin{equation*}
  \frac 1\pi\int_0^{2\pi} v_{k,\perp} v_{k,\star} \sin (n_\star s)
  = \frac 12 a_{k,0}b_{k,n_\star} - \frac 12 b_{k,n_\star}a_{k,2n_\star}
  +\frac 12 a_{k,n_\star}b_{k,2n_\star} .
\end{equation*}
This implies that
\begin{align*}
  \|f_{k,n_\star}^{\rm I}\|_{L^2}+\|g_{k,n_\star}^{\rm I}\|_{L^2}
  &\lesssim \|\Uk a_{k,0} a_{k,n_\star} \|_{L^2}
  + \|\Uk  a_{k,n_\star}a_{k,2n_\star} \|_{L^2}
  +\|\Uk  b_{k,n_\star} b_{k,2n_\star}\|_{L^2}\\
  &\quad   +\|\Uk a_{k,0} b_{k,n_\star} \|_{L^2} +\|\Uk  b_{k,n_\star} a_{k,2n_\star}\|_{L^2}
  +\|\Uk a_{k,n_\star} b_{k,2n_\star} \|_{L^2}\\
&\quad   +\left\|\Uk  v_{k,\perp}^2 \right\|_{L^2_sL^2_y}.
\end{align*}
Using Lemma \ref{le:bv} and Lemma \ref{le:vc} (with $n_\star\neq 0$), we have
\begin{equation*}
  \left\|\Uk  a_{k,0} a_{k,n_\star} \right\|_{L^2}
  \lesssim  \alpha_k^{-\frac 12} \|\UU\|_{L^2} \|a_{k,0}\|_{L^\infty}\|a_{k,n_\star}\|_{L^\infty}
  \lesssim \alpha_k^\frac12,
\end{equation*}
and by the same argument,
\begin{align*}
&\|\Uk a_{k,0} a_{k,n_\star} \|_{L^2}
  + \|\Uk  a_{k,n_\star}a_{k,2n_\star} \|_{L^2}
  +\|\Uk  b_{k,n_\star} b_{k,2n_\star}\|_{L^2}\\
  &\quad   +\|\Uk a_{k,0} b_{k,n_\star} \|_{L^2} +\|\Uk  b_{k,n_\star} a_{k,2n_\star}\|_{L^2}
  +\|\Uk a_{k,n_\star} b_{k,2n_\star} \|_{L^2}
  \lesssim \alpha_k^\frac12.
\end{align*}
Besides, by Lemma \ref{le:vp},
\begin{align*}
  \left\|\Uk  v_{k,\perp}^2 \right\|_{L^2_sL^2_y} 
  &\lesssim \alpha_k^{-1} \|\UU\|_{L^\infty} \|v_{k,\perp}\|_{L^4_sL^4_y}^2\\
%  &\lesssim \alpha_k^{-1} \int_0^{2\pi} 
%  \left(\|v_{k,\perp}(s)\|_{L^\infty}^2 \|v_{k,\perp}(s)\|_{L^2_y}^2 \right) ds\\
  &\lesssim \alpha_k^{-1} \|v_{k,\perp}\|_{L^4_sL^\infty_y} \|v_{k,\perp}\|_{L^4_sL^2_y}\\
  &\lesssim \alpha_k^{-1} \|v_{k,\perp}\|_{L^4_sL^\infty_y} \|v_{k,\perp}\|_{L^\infty_sL^2_y}^\frac12
  \|v_{k,\perp}\|_{L^2_sL^2_y}^\frac12\lesssim \alpha_k^\frac12.
\end{align*}
Therefore, we have proved
\begin{equation*}
  \|f_{k,n_\star}^{\rm I}\|+\|g_{k,n_\star}^{\rm I}\|_{L^2} \lesssim \alpha_k^\frac12.
\end{equation*}

Second,
\begin{align*}
  \qqq_k^{\rm II} &= \frac 16 (v_{k,\star}+v_{k,\perp})^3
  =\frac 16 v_{k,\star}^3+  \frac 12 v_{k,\star}^2 v_{k,\perp} + \frac 12 v_{k,\star}^2 v_{k,\perp}
  +\frac 16 v_{k,\perp}^3\\
  & =\frac16 a_{k,n_\star}^3 \cos^3(n_\star s) + \frac 12 a_{k,n_\star}^2b_{k,n_\star} \cos^2(n_\star s) \sin(n_\star s)
    + \frac 12  a_{k,n_\star} b_{k,n_\star}^2 \cos(n_\star s) \sin^2(n_\star s)\\
  &\quad  + \frac 16 b_{k,n_\star}^3 \sin^3(n_\star s) 
   + \frac 12 v_{k,\star}^2 v_{k,\perp} + \frac 12 v_{k,\star}^2 v_{k,\perp}
  +\frac 16 v_{k,\perp}^3. 
\end{align*}
Thus, using 
\begin{equation}\label{eq:tg}
\begin{aligned}
  \cos^3(n_\star s) & = \frac 34 \cos(n_\star s) + \frac 14 \cos(3n_\star s) \\
  \cos^2(n_\star s) \sin(n_\star s) & = \frac 14 \sin(n_\star s) + \frac 14 \sin(3n_\star s) \\
  \cos(n_\star s) \sin^2(n_\star s) & = \frac 14 \cos(n_\star s) - \frac 14 \cos(3n_\star s) \\
  \sin^3(n_\star s) & = \frac 34 \sin(n_\star s) - \frac 14 \sin(3n_\star s) \\
\end{aligned}
\end{equation}
we obtain
\begin{align*}
  & \qqq_k^{\rm II}
  -\frac 18 a_{k,n_\star}(a_{k,n_\star}^2+b_{k,n_\star}^2) \cos(n_\star s)
  -\frac 18 b_{k,n_\star}(a_{k,n_\star}^2+b_{k,n_\star}^2) \sin(n_\star)\\
  & \quad =  \frac 1{24} a_{k,n_\star}  (a_{k,n_\star}^2-3 b_{k,n_\star}^2) \cos(3n_\star s) 
  + \frac 1{24} b_{k,n_\star}( 3 a_{k,n_\star}^2-b_{k,n_\star}^2) \sin(3n_\star s)\\
  & \qquad+ \frac 12 v_{k,\star}^2 v_{k,\perp} + \frac 12 v_{k,\star}^2 v_{k,\perp}
  +\frac 16 v_{k,\perp}^3.
\end{align*}
This implies that
\begin{align*}
  & \left\|f_{k,n_\star}^{\rm II}-\frac 18 a_{k,n_\star}(a_{k,n_\star}^2+b_{k,n_\star}^2)\right\|_{L^2}
  + \left\|g_{k,n_\star}^{\rm II}-\frac 18 b_{k,n_\star}(a_{k,n_\star}^2+b_{k,n_\star}^2)\right\|_{L^2}
  \\ & \quad \lesssim \|v_{k,\star}^2 v_{k,\perp}\|_{L^2_sL^2_y} 
  +\| v_{k,\star}^2 v_{k,\perp}\|_{L^2_sL^2_y} + \| v_{k,\perp}^3\|_{L^2_sL^2_y}.
\end{align*}
We estimate the terms on the right-hand side as follows.
By Lemmas \ref{le:bv} and \ref{le:vp},
\begin{align*}
  \|v_{k,\star}^2 v_{k,\perp}\|_{L^2_sL^2_y} &\lesssim 
  \|v_{k,\star}\|_{L^4_sL^\infty_y}^2 \|v_{k,\perp}\|_{L^\infty_sL^2_y} \lesssim \alpha_k,\\
  \|v_{k,\star} v_{k,\perp}^2\|_{L^2_sL^2_y} &\lesssim 
  \|v_{k,\star}\|_{L^4_sL^\infty_y} \|v_{k,\perp}\|_{L^4_sL^\infty_y} \|v_{k,\perp}\|_{L^\infty_sL^2_y} \lesssim \alpha_k^\frac32,\\
  \| v_{k,\perp}^3\|_{L^2_sL^2_y}&\lesssim 
  \|v_{k,\perp}\|_{L^4_sL^\infty_y}^2 \|v_{k,\perp}\|_{L^\infty_sL^2_y} \lesssim \alpha_k^2.
\end{align*}
Therefore,
\begin{equation*}
  \left\|f_{k,n_\star}^{\rm II}-\frac 18 a_{k,n_\star}(a_{k,n_\star}^2+b_{k,n_\star}^2)\right\|_{L^2}
  + \left\|g_{k,n_\star}^{\rm II}-\frac 18 b_{k,n_\star}(a_{k,n_\star}^2+b_{k,n_\star}^2)\right\|_{L^2}
  \lesssim\alpha_k.
\end{equation*}

Third, from \eqref{eq:pp}, it follows that
\begin{equation*}
  |\qqq_k^{\rm III}|\lesssim \alpha_k |v_k|^4,
\end{equation*}
and so by Lemma \ref{le:bv},
\begin{equation*}
  \|f_{k,n_\star}^{\rm III}\|_{L^2}+\|g_{k,n_\star}^{\rm III}\|_{L^2}
  \lesssim \|\qqq_k^{\rm III}\|_{L^2_sL^2_y}\lesssim 
  \alpha_k \|v_k\|_{L^8_sL^8_y}^4
  \lesssim \alpha_k \|v_k\|_{L^\infty_sL^2_y}\|v_k\|_{L^\infty_sL^\infty_y}\|v_k\|_{L^4_sL^\infty_y}^2
  \lesssim \alpha_k^\frac 12.
\end{equation*}
(We have used $\|v_k\|_{L^\infty_sL^\infty_y}^2\lesssim \|v_k\|_{L^\infty_sL^2_y}\|\partial_yv_k\|_{L^\infty_sL^2_y}
\lesssim \alpha_k^{-1}$.)
\end{proof}

\subsection{A compactness result}
We prove the following compactness result using techniques in \cite{BeCe,lions}.
\begin{proposition}\label{pr:dc}
Let $\lambda\in \R$.
If there exists a sequence $\{(a_k,b_k)\}_k$ of $H^2(\R)\times H^2(\R)$ functions 
which is bounded in $H^1$, such that
\begin{equation}\label{eq:ab}
\lim_{k\to+\infty} \left\{\Bigl\|a_k''-\lambda  a_k + \frac 18 a_k(a_k^2+b_k^2) \Bigr\|_{L^2} 
+ \Bigl\| b_k''-\lambda b_k + \frac 18 b_k(a_k^2+b_k^2) \Bigr\|_{L^2}\right\} 
= 0
\end{equation}
and
\[
\liminf_{k\to+\infty} \left(\|a_k\|_{L^\infty}+ \|b_k\|_{L^\infty}\right) \gtrsim 1,
\]
then $\lambda>0$. Moreover,  
there exists an integer $J\geq 1$ and for all $j\in \{1,\ldots, J\}$,
there exist $\theta_j\in [-1,1]$ and a sequence of
reals $\{r_{j,k}\}_k$,  such that
\begin{equation*}
\lim_{k\to+\infty} \left\{\|a_k -  W_k\|_{H^1}+\| b_k - Z_k\|_{H^1} \right\}=0
\end{equation*}
where
\begin{equation}\label{eq:Wk}
W_k(y)=\sum_{j=1}^J  \cos(\theta_j) \sqrt{\lambda} Q\bigl(\sqrt{\lambda} (y-r_{j,k})\bigr),
\quad
Z_k(y)=\sum_{j=1}^J  \sin(\theta_j)  \sqrt{\lambda} Q\bigl(\sqrt{\lambda} (y-r_{j,k})\bigr)
\end{equation}
and, if $J\geq 2$, for any $i,j\in \{1,\ldots,J\}$, $i\neq j$,
\begin{equation}\label{eq:rj}
\lim_{k\to +\infty} |r_{i,k}-r_{j,k}|=+\infty.
\end{equation}
\end{proposition}
\begin{proof}
We start with the following uniqueness result.
\begin{lemma}\label{le:mk}
Let $\lambda>0$.
Let $(A,B)\in H^1(\R)\times H^1(\R)$ be a nonzero solution of the system
\begin{equation}\label{eq:AB}
  \begin{cases}
    A''- \lambda A + \frac 18 A (A^2+B^2)=0\\
    B''- \lambda B + \frac 18 B (A^2+B^2)=0
  \end{cases}
\end{equation}
Then, there exists $r\in\R$ and $\theta\in[0,2\pi)$ such that
\begin{align*}
  A(y)&=\cos(\theta) \sqrt{\lambda} Q(\sqrt{\lambda} ( y - r))\\
  B(y)&=\sin(\theta) \sqrt{\lambda} Q(\sqrt{\lambda} (y - r)).
\end{align*}
\end{lemma}
\begin{proof}[Proof of Lemma \ref{le:mk}]
First, we recall a general fact concerning the ODE
\begin{equation}\label{eq:DD}
  D''-\lambda D + V D=0
\end{equation}
with a potential $V(y)\in \mathcal{S}(\R)$.
Let $D_1$ be the solution of \eqref{eq:DD} with $D_1(0)=1$ and $D_1'(0)=0$,
and let $D_2$ be the solution of \eqref{eq:DD} with $D_1(0)=0$ and $D_2'(0)=1$.
Then there exists at most one value $\nu\in[-1,1]$ such that 
the solution $D=\nu D_1 + \sqrt{1-\nu^2} D_2$ belongs to $H^1(\R)$.
Indeed, assume for the sake of contradiction that two different $\nu_1$ and
$\nu_2$ in $[-1,1]$ provide an $H^1(\R)$ solution of \eqref{eq:DD}.
Then, $D_1$ and $D_2$ are both $H^1$ solution of \eqref{eq:DD}.
But this is a contradiction with the fact that the Wronskian
$W(y)=D_1(y)D_2'(y)-D_2(y)D_1'(y)$ is constant and equal to~$1$ on $\R$.

Second, we consider $(A,B)\in H^1(\R)\times H^1(\R)$ as in the statement of the lemma
and we set $V(y)=\frac18 (A^2+B^2)$.  Since $(A,B)$ is nontrivial, either $A$ or $B$ is nonzero.
We assume without loss of generality that $A$ is not zero.
By the observation above, there exists $\kappa\in\R$ such that $B=\kappa A$.

Third, we insert $B=\kappa A$ in the first line of the system and we get
\begin{equation*}
  A''- \lambda A + \frac{1+\kappa^2} 8 A^3.
\end{equation*}
By standard ODE tools, this implies that there exists $r\in \R$ such that
\begin{equation*}
A(y) = \nu \sqrt{\lambda} Q(\sqrt{\lambda}( y -r)),\quad
\nu = \frac{\pm 1}{\sqrt{1+\kappa^2}}.
\end{equation*}
Thus, setting $\sigma  = \SIGN(\kappa)$, we obtain
$B=\kappa A = \sigma \sqrt{1-\nu^2}\sqrt{\lambda}Q(\sqrt{\lambda} (y -r))$.
Lastly, we choose the unique $\theta\in [0,2\pi)$ such that
$\cos(\theta)=\nu$ and $\sin(\theta)=\sigma\sqrt{1-\nu^2}$.
\end{proof}

Now, we derive Proposition \ref{pr:dc} using Lemma \ref{le:mk} and following the proof of \cite[Lemma 3.1]{BeCe}.
Since we assume $\|a_k\|_{L^\infty}+\|b_k\| _{L^\infty} \gtrsim 1$,
there exists $c_1>0$ and 
for all $k\geq 1$, there exists
 $y_{1,k}\in \R$ be such that
\[
|a_k(y_{1,k})|+|b_k(y_{1,k})| \geq c_1.
\]
We set $a_{1,k}(y) = a_k(y+y_{1,k})$, $b_{1,k}(y) = b_k(y+y_{1,k})$.
Extracting a subsequence (but keeping the same notation), there exists $(A_1,B_1)\in H^1(\R)\times H^1(\R)$ such that
\begin{align*}
&  a_{1,k}\rightharpoonup A_1 \mbox{ weakly in $H^1(\R)$},
\quad  a_{1,k}\to A_1 \mbox{ a.e. on $\R$, as $k\to +\infty$,}\\
&  b_{1,k}\rightharpoonup B_1 \mbox{ weakly in $H^1(\R)$},
\quad  b_{1,k}\to B_1 \mbox{ a.e. on $\R$, as $k\to +\infty$.}
\end{align*}
By the $H^1$ bound, there exists $\delta>0$ such that for all $k$ sufficiently large and all $y\in [-\delta,\delta]$,
$| a_{1,k}(y)|+| b_{1,k}(y)|\geq \frac 12 c_1$.
Moreover, by the $H^1$ bound and Ascoli's theorem, $ a_{1,k}\to A_1$
and $ b_{1,k}\to B_1$ uniformly on $[-\delta,\delta]$
(possibly extracting a subsequence).
Therefore, $|A_1|+|B_1|\neq 0$.

Besides, $(A_1,B_1)$ satisfies the system \eqref{eq:AB}.
In particular, $A_1,B_1\in H^s(\R)$ for any $s\geq 0$.
Multiplying the first line of the system by $A_1' \tanh(y)$ and integrating by parts, 
multiplying the second line of the system by $A_2' \tanh(y)$ and integrating by parts,
then summing, we obtain
\[
\int \left((A_1')^2 +(B_1')^2- \lambda  (A_1^2+B_1^2)  
+ \frac 1{16} (A_1^4+B_1^4) \right) \frac {dy}{\cosh^2(y)} = 0
\]
which implies that $\lambda>0$.
Once this property is established, the rest of the proof closely follows 
the one in \cite[Lemma 3.1]{BeCe} using Lemma \ref{le:mk}.
We provide a proof for the reader convenience.

Since $(A_1,B_1)\neq (0,0)$ and satisfies the system \eqref{eq:AB}, by Lemma \ref{le:mk},
there exists $r_1^0\in \R$ and $\theta_1\in [0,2\pi)$ such that
\begin{align*}
  A_1(y)&=\cos(\theta_1) \sqrt{\lambda} Q(\sqrt{\lambda} (y - r_1^0))\\
  B_1(y)&=\sin(\theta_1) \sqrt{\lambda} Q(\sqrt{\lambda} (y - r_1^0)).
\end{align*}
Using \eqref{eq:ab} and integration by parts, we have
\begin{equation}\label{eq:ka}
  \lim_{k\to+\infty} \int \left[(a_k')^2+(b_k')^2 + \lambda a_k^2 + \lambda b_k^2 \right]
  - \frac 18 \int \bigl( a_k^2 + b_k^2 \bigr)^2 = 0.
\end{equation}
Using the equation of $Q$, it holds
\begin{equation*}
  \int (Q')^2 + Q^2 = \frac 18 \int Q^4
\end{equation*}
and thus (or using the system satisfied by $(A_1,B_1)$), it also holds
\begin{equation}\label{eq:o1}
  \int \left[(A_1')^2+(B_1')^2 + \lambda A_1^2 + \lambda B_1^2 \right]
  = \frac 18 \int (A_1^2 + B_1^2)^2 = \frac{\lambda^2}8 \int Q^4.
\end{equation}
Set
\[
 a_{2,k}(y)=a_{1,k}(y)-A_1(y),\quad b_{2,k}(y)=b_{1,k}(y)-B_1(y).
\]

If
\[
\lim_{k\to+\infty} \int (a_{2,k}')^2+(b_{2,k}')^2+\lambda a_{2,k}^2+\lambda b_{2,k}^2 =0
\]
then the proof is finished with $J=1$.
Otherwise, up to a subsequence, we assume
\[
\lim_{k\to+\infty} \int (a_{2,k}')^2+(b_{2,k}')^2+\lambda a_{2,k}^2+\lambda b_{2,k}^2 = c>0.
\]
We have
\begin{align*}
&  a_{2,k}\rightharpoonup 0 \mbox{ weakly in $H^1(\R)$},
\quad  a_{2,k}\to 0 \mbox{ a.e. on $\R$, as $k\to +\infty$,}\\
&  b_{2,k}\rightharpoonup 0 \mbox{ weakly in $H^1(\R)$},
\quad  b_{2,k}\to 0 \mbox{ a.e. on $\R$, as $k\to +\infty$.}
\end{align*}
Moreover, by weak convergence in $H^1$
\begin{equation}\label{eq:o2}
  \begin{aligned}
  \lim_{k\to+\infty} \int (a_{2,k}')^2+(b_{2,k}')^2+\lambda a_{2,k}^2+\lambda b_{2,k}^2 
  &=\lim_{k\to+\infty} \int (a_{k}')^2+(b_{k}')^2+\lambda a_{k}^2+\lambda b_{k}^2 \\
  &\quad - \int \left[(A_1')^2+(B_1')^2 + \lambda A_1^2 + \lambda B_1^2 \right].
\end{aligned}
\end{equation}
By a variant of the Brezis-Lieb lemma,  we also have
\begin{equation*}
  \lim_{k\to+\infty} \int (a_{2,k}^2 + b_{2,k}^2)^2 =
  \lim_{k\to+\infty} \int (a_{k}^2 + b_{k}^2)^2 -\int (A_1^2+B_1^2)^2.
\end{equation*}
Therefore, using \eqref{eq:ka} and \eqref{eq:o1}, we obtain
\begin{equation}\label{eq:I2}
  \lim_{k\to+\infty} \int \left[(a_{2,k}')^2+(b_{2,k}')^2 + \lambda a_{2,k}^2 + \lambda b_{2,k}^2 \right]
  =\lim_{k\to+\infty} \frac 18 \int \bigl( a_{2,k}^2 + b_{2,k}^2 \bigr)^2 = c >0.
\end{equation}
Set $J_m=[m,m+1]$ where $m\in \Z$ and 
\begin{equation*}
  \beta_k=\max_{m\in \Z} \left(\int_{J_m} (a_{2,k}^4+b_{2,k}^4)\right)^\frac14.
\end{equation*}
We estimate, using the Gagliardo-Nirenberg inequality on the interval $[0,1]$,
\begin{align*}
  \int (a_{2,k}^2+b_{2,k}^2)^2 
  &\lesssim \sum_{m\in\Z} \int_{J_m} (a_{2,k}^4+b_{2,k}^4)\\
  &\lesssim \beta_k^2 \sum_{m\in\Z} \left(\int_{J_m} (a_{2,k}^4+b_{2,k}^4)\right)^\frac 12\\
  &\lesssim \beta_k^2  \int(a_{2,k}')^2+(b_{2,k}')^2+\lambda a_{2,k}^2+\lambda b_{2,k}^2 .
\end{align*}
Thus, by \eqref{eq:I2}, possibly extracting a further subsequence, we obtain
$\lim_{k\to+\infty} \beta_k >0$.
For $k$ large, fix $m_k$ such that
\begin{equation*}
  \beta_k= \biggl(\int_{J_{m_k}} (a_{2,k}^4+b_{2,k}^4)\biggr)^\frac14.
\end{equation*}
We claim that $\lim_{k\to+\infty} |m_k|=+\infty$.
Indeed, otherwise, up to the extraction of a subsequence $m_k=m_0$ for all $k$
and $\lim_{k\to+\infty} \int_{J_{m_0}} (a_{2,k}^4+b_{2,k}^4)>0$.
But this is in contradiction with the weak convergence to $0$ in $H^1\times H^1$
of the sequence $(a_{2,k},b_{2,k})_k$.

Extracting a subsequence, there exists $(A_2,B_2)\in H^1(\R)\times H^1(\R)$ such that
\begin{align*}
&  a_{2,k}(\cdot+m_k)\rightharpoonup A_2 \mbox{ weakly in $H^1(\R)$},
\quad  a_{2,k}(\cdot+m_k)\to A_2 \mbox{ a.e. on $\R$, as $k\to +\infty$,}\\
&  b_{2,k}(\cdot+m_k)\rightharpoonup B_2 \mbox{ weakly in $H^1(\R)$},
\quad  b_{2,k}(\cdot+m_k)\to B_2 \mbox{ a.e. on $\R$, as $k\to +\infty$.}
\end{align*}
By the above estimates, we know that $(A_2,B_2)\neq (0,0)$.

Now, we look at the limiting system for $(A_2,B_2)$.
We have
\begin{align*}
  a_{2,k}''-\lambda  a_{2,k} + \frac 18 a_{2,k}(a_{2,k}^2+b_{2,k}^2)
  & = \left( a_{1,k}''-\lambda  a_{1,k} + \frac 18 a_{1,k}(a_{1,k}^2+b_{1,k}^2) \right)\\
  &  \quad - \left( A_1''-\lambda A_1 + \frac 18 A_1 (A_1^2+B_1^2) \right) +\frac 18 F_k
\end{align*}
where 
\begin{align*}
  F_k & = (a_{1,k}-A_1) \left((a_{1,k}-A_1)^2 + (b_{1,k}-B_1)^2\right) 
  - a_{1,k}(a_{1,k}^2+b_{1,k}^2) + A_1 (A_1^2+B_1^2)\\
  & = 3a_{1,k} A_1(A_1-a_{1,k}) + a_{1,k} B_1 (B_1-b_{1,k})
  + b_{1,k}B_1 (A_1-a_{1,k}) +A_1b_{1,k} (B_1-b_{1,k}) \\
  & = (3a_{1,k} A_1+ b_{1,k}B_1 ) a_{2,k} + (a_{1,k} B_1+A_1b_{1,k}) b_{2,k}.
\end{align*}
By the last expression of $F_k$, it is easy to see that $\lim_{k\to+\infty}\|F_k\|_{L^2}=0$.
Using \eqref{eq:ab} and the equation
$A_1''-\lambda A_1 + \frac 18 A_1 (A_1^2+B_1^2)=0$, we deduce that
\begin{equation*}
  \lim_{k\to+\infty} \Bigl\|a_{2,k}''-\lambda  a_{2,k} + \frac 18 a_{2,k}(a_{2,k}^2+b_{2,k}^2)\Bigr\|_{L^2}=0.
\end{equation*}
Similarly, one proves that
\begin{equation*}
  \lim_{k\to+\infty} \Bigl\|b_{2,k}''-\lambda  b_{2,k} + \frac 18 b_{2,k}(a_{2,k}^2+b_{2,k}^2)\Bigr\|_{L^2}=0.
\end{equation*}
As before, this implies that $(A_2,B_2)\neq (0,0)$ satisfies the system \eqref{eq:AB}.
Thus, by Lemma \ref{le:mk}, there exists $r_2^0\in \R$ and $\theta_2\in [0,2\pi)$ such that
\begin{align*}
  A_2(y)&=\cos(\theta_2) \sqrt{\lambda} Q(\sqrt{\lambda}( y - r_2^0))\\
  B_2(y)&=\sin(\theta_2) \sqrt{\lambda} Q(\sqrt{\lambda}( y - r_2^0)).
\end{align*}

Now, we set
\[
 a_{3,k}(y)=a_{2,k}(y)-A_2(y),\quad b_{3,k}(y)=b_{2,k}(y)-B_2(y)
\]
and we iterate the argument.
For a certain $J\geq 3$, one finds that 
$(a_{J,k},b_{J,k})$ converges strongly to $0$ in $H^1\times H^1$.
Indeed, as shown by \eqref{eq:o1} and \eqref{eq:o2}, 
measuring $(a_{j,k},b_{j,k})$ in the norm
\[
\left(\int \left[(A')^2+(B')^2+\lambda A^2 +\lambda B^2\right]\right)^\frac12,
\]
each step passing
from $(a_{j-1,k},b_{j-1,k})$ to $(a_{j,k},b_{j,k})$ removes a fixed amount $\frac{\lambda^2}{8}\int Q^4$ of this norm
and since the initial sequence $(a_{1,k},b_{1,k})$ has a finite norm, this can only be
done a finite number of times.
\end{proof}

\subsection{Convergence of the dominant mode}
In this subsection, we complete the proof of \ref{th:i} and \ref{th:ii} of Theorem \ref{thm:1},
using Proposition \ref{pr:dc}.

From the system \eqref{eq:vs} of $(a_{k,n_\star},b_{k,n_\star})$, Lemma \ref{le:6} and the convergence $\lim_{k\to+\infty} \lambda_{k,n_*} = \lambda$, we have
\begin{equation}\label{eq:av}
  \begin{aligned}
  \lim_{k\to+\infty} \left\|a_{k,n_\star}''-\lambda  a_{k,n_\star} + \frac 18 a_{k,n_\star}
  (a_{k,n_\star}^2+b_{k,n_\star}^2) \right\|_{L^2} &= 0,\\
  \lim_{k\to+\infty} \left\|b_{k,n_\star}''-\lambda  b_{k,n_\star} + \frac 18 b_{k,n_\star}
  (a_{k,n_\star}^2+b_{k,n_\star}^2) \right\|_{L^2} &= 0.
  \end{aligned}
\end{equation}
Moreover, by Lemmas \ref{le:bv} and \ref{le:vp}, we have
\begin{equation*}
  \|a_{k,n_\star}\|_{H^1}+\|b_{k,n_\star}\|_{H^1} \lesssim 1 \quad\mbox{and}\quad 
  \|a_{k,n_\star}\|_{L^\infty}+\|b_{k,n_\star}\|_{L^\infty} \gtrsim 1.
\end{equation*}
Applying Proposition \ref{pr:dc} to the sequence of dominant modes $\{(a_{k,n_\star},b_{k,n_\star})\}_k$,
we deduce that
$\lambda>0$ and that there exists
an integer $J\geq 1$ and for all $j\in \{1,\ldots, J\}$,
there exist $\theta_j\in [0,2\pi)]$ and a sequence of
reals $\{r_{j,k}\}_k$ such that
\begin{equation}\label{eq:zz}
\lim_{k\to+\infty} \left\{\|a_{k,n_\star} -  W_k\|_{H^1}+\| b_{k,n_\star} - Z_k\|_{H^1} \right\}=0
\end{equation}
where $\{W_k\}_k$, $\{Z_k\}_k$ and $\{r_{j,k}\}_k$ are as in \eqref{eq:Wk} and \eqref{eq:rj}.

We introduce the notation 
\[
\omega_k = \frac{T_\star}{T_k}=\frac{2\pi n_\star}{T_k}.
\]
By \eqref{eq:Tl} and $\lambda>0$, $\lim_{k\to +\infty} \omega_k=1^-$
and
\[
\lim_{k\to+\infty} \frac{1-\omega_k^2}{\alpha_k^2} = \lambda.
\]
In particular
\[
1-\omega_k = \frac{\lambda}{1+\omega_k}  \alpha_k^2+ o(\alpha_k^2)
\quad \mbox{and so}\quad
\frac{T_k}{2\pi n_\star} = \frac1{\omega_k}
=1+\frac{\lambda }2\alpha_k^2+o(\alpha_k^2).
\]
Thus, \ref{th:i} of Theorem \ref{thm:1} is proved.

To prove \ref{th:ii} of Theorem \ref{thm:1}, we go back to the original variables
\[
\phi_k(t,x)=\alpha_k v_k\left(\frac{2\pi}{T_k} t,\alpha_k x\right)
\]
where from Lemma \ref{le:vp}, we have the decomposition
$v_k = v_{k,\star}+v_{k,\perp}$ with
\[
v_{k,\star}(s,y)=a_{k,n_\star}(y) \cos(n_\star s) + b_{k,n_\star}(y) \sin(n_\star s) ,\quad
\|v_{k,\perp}\|_{L^\infty_sL^2_y} \lesssim \alpha_k.
\]
We observe that
\begin{align*}
& \phi_k(t,x)-\alpha_k [W_k(\alpha_k x) \cos(\omega_k t)+ Z_k(\alpha_k x) \sin(\omega_k t)]\\
&\quad = \alpha_k \left[ (a_{k,n_\star}-W_k)(\alpha_k x) \cos(\omega_k t)
+(b_{k,n_\star}-Z_k)(\alpha_k x) \sin(\omega_k t)
+ v_{k,\perp}\left(\frac {2\pi}{T_k} t,\alpha_k x\right)\right]
\end{align*}
and thus, for all $t\in [0,T]$,
\begin{align*}
&\alpha_k^{-1}\int |\phi_k(t,x)-\alpha_k [W_k(\alpha_k x) \cos(\omega_k t)+Z_k(\alpha_k x) \sin(\omega_k t)]|^2 \ud x \\
&\quad \lesssim \|a_{k,n_\star}-W_k\|_{L^2}^2
+\|b_{k,n_\star}-Z_k\|_{L^2}^2 + \|v_{k,\perp}\|_{L^\infty_sL^2_y}^2.
\end{align*}
Using \eqref{eq:zz} and $\|v_{k,\perp}\|_{L^\infty_sL^2_y} \lesssim \alpha_k$, this implies that
\[
\lim_{k\to+\infty} \alpha_k^{-1} \sup_{t\in\R}
\int \left|\phi_k(t,x)-\alpha_k [W_k(\alpha_k x) \cos(\omega_k t)+ Z_k(\alpha_k x) \sin(\omega_k t)]\right|^2 \ud x =0.
\]
Using $\cos(\theta_j)\cos(\omega_k t) + \sin(\theta_k)\sin(\omega_k t)
=\cos(\omega_k t - \theta_j)$, we have
\[
W_k(\alpha_k x) \cos(\omega_k t)+Z_k(\alpha_k x) \sin(\omega_k t)
= \sum_{j=1}^J \sqrt{\lambda} Q(\sqrt{\lambda} (\alpha_k x-r_{j,k}) ) \cos(\omega_k t - \theta_j).
\]
Defining
\[
\varepsilon_k = \sqrt{1-\omega_k^2} >0
\]
so that $\lim_{k\to+\infty} \frac{\varepsilon_k}{\alpha_k} = \sqrt{\lambda}$,
we have proved \ref{th:ii} of Theorem \ref{thm:1}.

\begin{remark}
  Using $\lim_{k\to+\infty} \|\partial_s v_{k,\perp}\|_{L^2_sL^2_y} = 0$
  from Lemma \ref{le:vp}, we also deduce that
\[
\lim_{k\to+\infty} \varepsilon_k^{-1} \int_0^{2\pi}
\int \biggl|\partial_t \phi_k(t,x)+ \varepsilon_k\omega_k \sum_{j=1}^J  Q( \varepsilon_k x-r_{j,k}) 
\sin(\omega_k t - \theta_j)\biggr|^2 \ud x \ud t=0.
\]
\end{remark}

\subsection{The Fermi golden rule}
In this subsection, we prove that if for a certain $j\in \{1,\ldots,J\}$
the sequence $\{r_{j,k}\}_k$ is bounded then $\widehat \UU(\sqrt{3})=0$.
This will imply \ref{th:iii} of Theorem \ref{thm:1}.

Without loss of generality, we assume that the sequence $\{r_{1,k}\}_k$ is bounded.
Thus, up to the extraction of a subsequence,
there exists $r_1=\lim_{k\to\infty} r_{1,k}$.
We have to examine the mode $2n_\star$.
From \eqref{eq:vn}, we have
\begin{equation}\label{eq:v2}
  \begin{aligned}
  a_{k,2n_\star}'' - \lambda_{k,2n_\star} a_{k,2n\star} + f_{k,2n_\star} = 0\\
  b_{k,2n_\star}'' - \lambda_{k,2n_\star} b_{k,2n\star} + g_{k,2n_\star} = 0
  \end{aligned}
\end{equation}
Recall that $\lambda_{k,2n_\star}<0$ and $\lambda_{k,2n_\star} \sim -3/\alpha_k^2$
as $k\to+\infty$.
Since $a_{k,2n_\star}, b_{k,2n_\star}\in H^1$, we have
\begin{equation}\label{eq:cc}
\int  \cos(\sqrt{-\lambda_{k,2n_\star}} y) f_{k,2n_\star}(y) \ud y=
\int  \cos(\sqrt{-\lambda_{k,2n_\star}} y) g_{k,2n_\star}(y) \ud y=0.
\end{equation}
We use the same notation as in the proof of Lemma \ref{le:6}.
Recall from \eqref{eq:fk} that
\begin{equation}\label{eq:fk3}
  \qqq_k =\Uk  v_k^2
  +\frac16 v_k^3 + \frac1{\alpha_k^3}\pp(\alpha_k v_k) =q_{k}^{\rm I}+q_{k}^{\rm II}+q_{k}^{\rm III}.
\end{equation}
First, expanding $v_k = v_{k,\star}+v_{k,\perp}$, one has
\begin{equation*}
  \qqq_k^{\rm I} =\Uk  \left(v_{k,\star}^2+2v_{k,\star}v_{k,\perp}
  +v_{k,\perp}^2\right).
\end{equation*}
Recall
\begin{equation*}
  v_{k,\star}^2(s) =\frac 12 (a_{k,n_\star}^2+b_{k,n_\star}^2)
  + \frac 12 (a_{k,n_\star}^2-b_{k,n_\star}^2) \cos(2 n_*s)
  + a_{k,n_\star} b_{k,n_\star}\sin (2 n_\star s).
\end{equation*}
This implies that
\begin{equation*}
   \left\|f_{k,2n_\star}^{\rm I}-\frac 12\Uk 
   (a_{k,n_\star}^2-b_{k,n_\star}^2)\right\|_{L^1} 
    \lesssim \left\|\Uk  v_{k,\star} v_{k,\perp} \right\|_{L^1_sL^1_y}
  +\left\|\Uk  v_{k,\perp}^2 \right\|_{L^1_sL^1_y}.
\end{equation*}
We have 
\begin{equation*}
  \left\|\Uk  v_{k,\star} v_{k,\perp} \right\|_{L^1_sL^1_y}
  \lesssim   \|\UU\|_{L^1} \|v_{k,\perp}\|_{L^1_sL^\infty_y}\|v_{k,\star}\|_{L^\infty_sL^\infty_y}.
\end{equation*}
Using Lemma \ref{le:bv}, one has
$\|v_{k,\star}\|_{L^\infty_sL^\infty_y}\lesssim \|a_{k,n_\star}\|_{L^\infty}
+\|b_{k,n_\star}\|_{L^\infty}\lesssim \|v_k\|_{L^1_sL^\infty_y}\lesssim \|v_k\|_{L^2_sL^\infty_y}\lesssim1$,
and using Lemma \ref{le:vp}, one has $\|v_{k,\perp}\|_{L^1_sL^\infty_y}
\lesssim \|v_{k,\perp}\|_{L^4_sL^\infty_y}\lesssim\alpha_k^\frac12$.
Thus,
\begin{equation*}
  \left\|\Uk  v_{k,\star} v_{k,\perp} \right\|_{L^1_sL^1_y}
  \lesssim  \alpha_k^\frac12.
\end{equation*}
Moreover, by Lemma \ref{le:vp},
\begin{equation*}
  \left\|\Uk  v_{k,\perp}^2 \right\|_{L^1_sL^1_y}
  \lesssim \alpha_k^{-1}\|\UU\|_{L^\infty}  \|v_{k,\perp}\|_{L^2_sL^2_y}^2 \lesssim \alpha_k.
\end{equation*}
Therefore,
\begin{equation*}
  \left\|f_{k,2n_\star}^{\rm I}-\frac 12\Uk  
  (a_{k,n_\star}^2-b_{k,n_\star}^2)\right\|_{L^1}\lesssim \alpha_k^\frac12.
\end{equation*}
Similarly, we prove that
\begin{equation*}
  \left\|g_{k,2n_\star}^{\rm I}- \Uk  
  a_{k,n_\star}b_{k,n_\star}\right\|_{L^1}\lesssim \alpha_k^\frac12.
\end{equation*}

Second,
\begin{align*}
\qqq_k^{\rm II} 
&= \frac 16 (v_{k,\star}+v_{k,\perp})^3\\
&=\frac16a_{k,n_\star}^3 \cos^3(n_\star s) 
+\frac12a_{k,n_\star}^2b_{k,n_\star} \cos^2(n_\star s)\sin(n_\star s)
+\frac12a_{k,n_\star}b_{k,n_\star}^2 \cos(n_\star s)\sin^2(n_\star s)\\
&\quad
+\frac16b_{k,n_\star}^3 \sin^3(n_\star s)
 + \frac 12 v_{k,\star}^2 v_{k,\perp} + \frac 12 v_{k,\star}^2 v_{k,\perp}+\frac 16 v_{k,\perp}^3.
\end{align*}
Thus, using \eqref{eq:tg},
\begin{align*}
  \qqq_k^{\rm II} &=\frac 18 a_{k,n_\star}(a_{k,n_\star}^2+b_{k,n_\star}^2) \cos(n_\star s) 
  +\frac 1{24} a_{k,n_\star}(a_{k,n}^2-3 b_{k,n_\star}^2)\cos(3n_\star s) \\
&\quad  +\frac18b_{k,n_\star} (a_{k,n_\star}^2+b_{k,n_\star}^2) \sin(n_\star s)
  +\frac1{24}b_{k,n_\star}(3a_{k,n_\star}^2-b_{k,n_\star}^2) \sin(3n_\star s)\\
  &\quad + \frac 12 v_{k,\star}^2 v_{k,\perp} + \frac 12 v_{k,\star}^2 v_{k,\perp}
  +\frac 16 v_{k,\perp}^3.
\end{align*}
This implies that
\begin{equation*}
  \left\|f_{k,2n_\star}^{\rm II}\right\|_{L^1}+\left\|g_{k,2n_\star}^{\rm II}\right\|_{L^1}\lesssim \|v_{k,\star}^2 v_{k,\perp}\|_{L^1_sL^1_y} 
  +\| v_{k,\star}^2 v_{k,\perp}\|_{L^1_sL^1_y} + \| v_{k,\perp}^3\|_{L^1_sL^1_y}.
\end{equation*}
We estimate the terms on the right-hand side as follows.
By Lemmas \ref{le:bv} and \ref{le:vp},
\begin{align*}
  \|v_{k,\star}^2 v_{k,\perp}\|_{L^1_sL^1_y} &\lesssim 
  \|v_{k,\star}\|_{L^2_sL^\infty_y}\|v_{k,\star}\|_{L^\infty_sL^2_y}\|v_{k,\perp}\|_{L^\infty_sL^2_y} \lesssim \alpha_k,\\
  \|v_{k,\star} v_{k,\perp}^2\|_{L^1_sL^1_y} & \lesssim 
  \|v_{k,\star}\|_{L^2_sL^\infty_y}  \|v_{k,\perp}\|_{L^\infty_sL^2_y}^2 \lesssim \alpha_k^2,\\
  \| v_{k,\perp}^3\|_{L^1_sL^1_y}& \lesssim 
  \|v_{k,\perp}\|_{L^2_sL^\infty_y} \|v_{k,\perp}\|_{L^\infty_sL^2_y}^2 \lesssim \alpha_k^2.
\end{align*}
Therefore,
\begin{equation*}
  \left\|f_{k,2 n_\star}^{\rm II}\right\|_{L^1}+\left\|g_{k,2 n_\star}^{\rm II}\right\|_{L^1}\lesssim\alpha_k.
\end{equation*}

Third, from \eqref{eq:pp}, it follows that
\begin{equation*}
  |\qqq_k^{\rm III}|\lesssim \alpha_k  |v_k|^4,
\end{equation*}
and so by Lemma \ref{le:bv},
\begin{equation*}
  \|f_{k,n_\star}^{\rm III}\|_{L^1}+\|g_{k,n_\star}^{\rm III}\|_{L^1}\lesssim \|\qqq_k^{\rm III}\|_{L^1_sL^1_y}
  \lesssim \alpha_k \|v_k\|_{L^4_sL^4_y}^4
  \lesssim \alpha_k \|v_k\|_{L^\infty_sL^2_y}^2 \|v_k\|_{L^2_sL^\infty_y}^2
  \lesssim \alpha_k.
\end{equation*}
In conclusion of the estimates above, we have
\begin{equation*}
  \left\|f_{k,2n_\star} -\frac 12\Uk  
  (a_{k,n_\star}^2-b_{k,n_\star}^2)\right\|_{L^1}
  +\left\|g_{k,2n_\star} - \Uk  
  a_{k,n_\star}b_{k,n_\star}\right\|_{L^1}
  \lesssim \alpha_k^\frac12.
\end{equation*}

Combining the above estimates with \eqref{eq:cc}, we find
\begin{equation}\label{eq:ll}
  \left|  \int  \cos(\sqrt{-\lambda_{k,2n_\star}} y)\Uk
  (a_{k,n_\star}^2-b_{k,n_\star}^2) \right|+
  \left| \int  \cos(\sqrt{-\lambda_{k,2n_\star}} y)\Uk
   a_{k,n_\star} b_{k,n_\star} \right| \lesssim \alpha_k^\frac12.
\end{equation}
Changing variable, we have
\begin{align*}
&  \int\cos(\sqrt{-\lambda_{k,2n_\star}} y)
\Uk(y) (a_{k,n_\star}^2-b_{k,n_\star}^2)(y) \ud y\\
&\quad = \int \cos(\alpha_k \sqrt{-\lambda_{k,2n_\star}} x) \UU(x) (a_{k,n_\star}^2
-b_{k,n_\star}^2)(\alpha_k x) \ud x .
\end{align*}
Now, using \eqref{eq:zz}, $\lim_{k\to +\infty} \alpha_k^2 \lambda_{k,2n_\star}= - 3$
and \eqref{eq:ll}, we have
\begin{align*}
  &\lim_{k\to+\infty}
  \int \cos(\alpha_k \sqrt{-\lambda_{k,2n_\star}} x) \UU(x) (a_{k,n_\star}^2
  -b_{k,n_\star}^2)(\alpha_k x)) dx\\
  &\quad=  \cos(2\theta_1)\lambda Q^2(\sqrt{\lambda} r_1)\int \cos(\sqrt{3} x) \UU(x) \ud x=0
\end{align*}
and similarly,
\begin{align*}
  &\lim_{k\to+\infty}
  2 \int \cos(\alpha_k \sqrt{-\lambda_{k,2n_\star}} x) \UU(x) (a_{k,n_\star} 
  b_{k,n_\star}) (\alpha_k x)) dx\\
  &\quad= \sin(2\theta_1)\lambda Q^2(\sqrt{\lambda} r_1)\int \cos(\sqrt{3} x) \UU(x) \ud x=0.
\end{align*}
When passing to the limit as $k\to+\infty$ above, we have
used that $U\in \mathcal{S}(\R)$ and if $J\geq 2$ that $\lim_{k\to +\infty} |r_{j,k}|=+\infty$
for $j\in \{2,\ldots,J\}$.

Therefore, independently of the value of $\theta_1$, 
we have obtained $\int \cos(\sqrt{3} x) \UU(x) \ud x=0$.
One proves similarly that $\int \sin(\sqrt{3} x) \UU(x) \ud x=0$.

\end{document}